\documentclass[a4paper,10pt]{amsart}
 \usepackage{amsfonts,mathrsfs}
 \usepackage{amsthm}
 \usepackage{amssymb}
 \usepackage{enumerate}
 \usepackage{tikz-cd}
 \usepackage{cleveref}
 
\usepackage{mathtools}

\DeclarePairedDelimiter\floor{\lfloor}{\rfloor}
 
\theoremstyle{definition}
\newtheorem{Def}{Definition}[section]
\newtheorem{ex}[Def]{Example}
\newtheorem{cns}[Def]{Construction}
\newtheorem{rem}[Def]{Remark}

\theoremstyle{plain}
\newtheorem{prop}[Def]{Proposition}
\newtheorem{thm}[Def]{Theorem}
\newtheorem*{thm*}{Theorem}
\newtheorem{lem}[Def]{Lemma}
\newtheorem{cor}[Def]{Corollary}
\newtheorem*{cor*}{Corollary}

\newtheorem*{con*}{Conjecture}
\newtheorem*{frag*}{Question}
\newtheorem*{verm*}{Vermutung}

\usepackage{mathrsfs}
\newcommand{\Sym}{\operatorname{Sym}}

\newcommand{\Hom}{\operatorname{Hom}} %External Hom
\newcommand{\End}{\operatorname{End}}

\newcommand{\GL}{\operatorname{{\mathbf GL}}}

\newcommand{\tr}{\operatorname{tr}}
\newcommand{\diag}{\operatorname{diag}}

\newcommand{\De}{\operatorname{D}}
\newcommand{\Or}{\operatorname{O}}
\newcommand{\Ma}{\operatorname{Ma}}
\newcommand{\Min}{\operatorname{Min}}
\newcommand{\cL}{{\mathcal L}}

\newcommand{\cO}{{\mathcal O}}

\newcommand{\fS}{{\mathfrak S}}

\newcommand{\R}{{\mathbb R}}

\newcommand\xqed[1]{%
  \leavevmode\unskip\penalty9999 \hbox{}\nobreak\hfill
  \quad\hbox{#1}}
\newcommand\demo{\xqed{$\triangle$}}

\DeclareMathOperator*{\Co}{C}

\title[Spectral LMI]{Spectral linear matrix inequalities}
\author{Mario Kummer}
\address{Technische Universit\"at Dresden, Fakult\"at Mathematik, Institut f\"ur Geometrie, Zellescher Weg 12-14, 01062 Dresden, Germany}
\email{mario.kummer@tu-dresden.de}
\thanks{The author has been supported by the DFG under Grant No.421473641.}
\newcommand{\comment}[1]{}

\begin{document}

\subjclass[2010]{Primary: 15A39, 90C22; Secondary: 20C30, 14P10}

\begin{abstract}
 We prove, under a certain representation theoretic assumption, that the set of real symmetric matrices, whose eigenvalues satisfy a linear matrix inequality, is itself a spectrahedron. The main application is that derivative relaxations of the positive semidefinite cone are spectrahedra. From this we further deduce statements on their Wronskians. These imply that Newton's inequalities, as well as a strengthening of the correlation inequalities for hyperbolic polynomials, can be expressed as sums of squares.
\end{abstract}
\maketitle

\section{Introduction}
A homogeneous polynomial $h \in \R[x_1,\ldots,x_n]$ is said to be \textit{hyperbolic} with respect to
$e \in \R^{n}$, if $h(e)>0$ and if for every 
$a \in \R^{n}$ the univariate polynomial $h(t e -a)$ in $t$ has only real roots.
The \textit{hyperbolicity cone} $\Co (h,e)$ of $h$ at $e$ is the set of all $a \in \R^{n}$ such all zeros of 
$h(t e -a)$ are nonnegative. Hyperbolicity cones are closed convex cones by \cite{Gar}. An instructive example of a polynomial that is hyperbolic with respect to $e$ is given by $\det A(x)$ where $$A(x):=x_1A_1+\ldots+x_nA_n$$ for real symmetric matrices $A_i$ with the property that $A(e)$ is positive definite. In this case, the hyperbolicity cone is defined by a {linear matrix inequality} (LMI): $$\Co(\det A(x),e)=\{a\in\R^n:\,A(a)\textrm{ is positive semidefinite}\}.$$ Such sets are called \emph{spectrahedral cones}. A major open problem in this context is:

\begin{con*}[Generalized Lax Conjecture]\label{con:glc}
 Hyperbolicity cones are spectrahedral.
\end{con*}

There is positive \cite{hv,detbez,matching} and negative \cite{Bra11,amini,lowerbounds} evidence for this conjecture. A direct application of Rolle's theorem shows that $$\De^k_eh:=\left(\sum_{i=1}^n e_i\cdot \frac{\partial }{\partial x_i}\right)^k h$$ is hyperbolic with respect to $e$ for all $k\leq\deg(h)$ if $h$ is. These hyperbolic polynomials are often called \emph{Renegar derivatives} as their geometric properties were first studied by Renegar \cite{renegar}. The Generalized Lax Conjecture would imply in particular that the hyperbolicity cone $\Co(\De^k_e\det A(x),e)$ is spectrahedral. In the case when $A(x)$ is a diagonal matrix, this was shown by Br\"and\'en \cite{BranEle} after Sanyal \cite{pol1} proved the case $k=1$ relying on results from \cite{matroids}. The latter was used by Saunderson \cite{firstderi} to solve the case $k=1$ for possibly nondiagonal $A(x)$. We will generalize this result to arbitrary $k$.

\begin{thm*}
 The hyperbolicity cone $\Co(\De^k_e\det A(x),e)$ is spectrahedral. The size of this spectrahedral representation is $\cO(d^{2k+2})$ when the size $d$ of $A(x)$ grows.
\end{thm*}

Note that it was already shown in \cite{projspec} that $\Co(\De^k_e\det A(x),e)$ has a representation as a \emph{spectrahedral shadow}, i.e., the image of a spectrahedral cone under a linear map.

Our above mentioned result will be a special case of a more general statement that we want to describe in the following. Let $S\subset\R^n$ be a convex symmetric set, i.e., a set that is invariant under every permutation of the variables. The associated \emph{spectral convex set} is defined as$$\Lambda(S)=\{A\in\Sym_2(\R^n):\, \lambda(A)\in S\}$$and was recently introduced and studied by Sanyal and Saunderson \cite{sanyal2020spectral}. Here $\lambda(A)$ denotes the vector of eigenvalues of a real symmetric matrix $A$. Among others, they show that if $S$ is a spectrahedral shadow, then $\Lambda(S)$ is a spectrahedral shadow as well \cite[Thm.~4.1]{sanyal2020spectral}. Furthermore, if $S$ is a polytope, then $\Lambda(S)$ is even a spectrahedron \cite[Thm.~3.3]{sanyal2020spectral}. Now let $h\in\R[x_1,\ldots,x_n]$ be hyperbolic with respect to $e=(1,\ldots,1)$ and assume that $h$ is symmetric. Then its hyperbolicity cone $\Co(h,e)$ is symmetric and the associated spectral convex set $\Lambda(\Co(h,e))$ is a hyperbolicity cone as well by \cite[Thm.~3.1]{Bauschke}. Thus the Generalized Lax Conjecture asserts in particular that $\Lambda(S)$ is a spectrahedral cone whenever $S\subset\R^n$ is a symmetric spectrahedral cone. Although we are not able to prove this statement in its full generality, we establish a sufficient representation theoretic criterion on the LMI representation of $S$ for $\Lambda(S)$ being a spectrahedral cone. This criterion applies to the LMI description of the hyperbolicity cone of elementary symmetric polynomials that was constructed in \cite{BranEle}. From this we then obtain the above result on the hyperbolicity cones of Renegar derivatives.

Hyperbolic polynomials satisfy several types of inequalities. One of those can be expressed in terms of the \emph{Wronskian polynomial}: For any $a,b\in\R^n$ the Wronskian polynomial $\Delta_{a,b}(h)$ of $h\in\R[x_1,\ldots,x_n]$ is defined as $$\Delta_{a,b}(h)=\De_ah\cdot\De_bh-h\cdot\De_a\De_bh.$$ If $h$ is hyperbolic with respect to $e$ and $a,b\in\Co(h,e)$, then the Wronskian $\Delta_{a,b}(h)$ is globally nonnegative on $\R^n$. This follows from \cite[Thm.~5.6]{petterwron} or \cite[Thm.~3.1]{multiaffine} and is sometimes called the \emph{correlation inequality}. In fact, one can sharpen this inequality to the following inequality which holds on all of $\R^n$:
 $$\Delta_{a,b}(h)\geq \frac{h(b)}{\De_ah(b)}\cdot(\De_ah)^2.$$
Using our spectrahedral representations, we prove that for the Renegar derivatives $\De^k_e\det A(x)$ this inequality can be expressed as a sum of squares. Choosing $h$ to be the elementary symmetric polynomial $\sigma_{d+1,n}\in\R[x_1,\ldots,x_n]$ of degree $d+1$ and $a=b$ to be the all-ones vector, this recovers exactly Newton's inequalities:

\begin{thm*}
 The polynomial $$\left(\frac{\sigma_{d,n}}{\binom{n}{d}}\right)^2-\left(\frac{\sigma_{d+1,n}}{\binom{n}{d+1}}\right)\cdot \left(\frac{\sigma_{d-1,n}}{\binom{n}{d-1}}\right)$$ is a sum of squares of polynomials.
\end{thm*}

This implies a previous result by Gao and Wagner \cite{wronsos} stating that$${\sigma_{d,n}}^2-{\sigma_{d+1,n}}\cdot {\sigma_{d-1,n}}$$is a sum of squares.

\section{Outline}
Consider a representation $V$ of the symmetric group $\fS_n$ and an $\fS_n$-linear map $\varphi:\R^n\to\Sym_2(V)$. The preimage of the positive semidefinite cone in $\Sym_2(V)$ under $\varphi$ is a spectrahedral cone which is invariant under the action of the symmetric group $\fS_n$ on $\R^n$. Conversely, every spectrahedral cone $S\subset\R^n$ that is invariant under the action of $\fS_n$ arises in that way. Indeed, if $A(x)=A(x_1,\ldots,x_n)$ is a linear matrix polynomial that describes $S$, then the block-diagonal matrix consisting of all blocks $\sigma(A(x))=A(x_{\sigma(1)},\ldots,x_{\sigma(n)})$ for $\sigma\in\fS_n$ is of the desired form.

For $S\subset\R^n$ a symmetric spectrahedral cone as above let $\Lambda(S)\subset\Sym_2(\R^n)$ be the set of all symmetric $n\times n$ matrices whose spectrum lies in $S$. By \cite[Thm.~3.1]{Bauschke} the set $\Lambda(S)$ is a hyperbolicity cone. We give a sufficient criterion under which $\Lambda(S)$ is even a spectrahedral cone. To this end, since $\Lambda(S)$ is invariant under the action of $\Or(n)$ on $\Sym_2(\R^n)$, we want to replace our $\fS_n$-linear map $\varphi$ by a suitable $\Or(n)$-linear map. In order to formulate the precise criterion we make the following definition; 
the terminology related to partitions used in this article will be introduced in \Cref{sec:symn}. 

\begin{Def}
 A representation of $\fS_n$ is \emph{short} if it consists only of such irreducible representations that correspond to partitions of length at most $2$.
\end{Def}

In \Cref{sec:preps} we will explicitly characterize all $\fS_n$-linear maps $\R^n\to\Sym_2(V)$ for short representations $V$ of $\fS_n$. Using this characterization, we will prove the following result  in \Cref{sec:mainres}:

\begin{thm}\label{thm:main}
 Let $V$ be a short representation of $\fS_n$ and $\varphi:\R^n\to\Sym_2(V)$ an $\fS_n$-linear map. Let $S\subset\R^n$ be the preimage of the positive semidefinite cone in $\Sym_2(V)$ under $\varphi$. Then $\Lambda(S)\subset\Sym_2(\R^n)$ is a spectrahedral cone.
\end{thm}

More precisely, we will associate to each short representation $V$ of $\fS_n$ a representation $W$ of $\Or(n)$ together with an $\fS_n$-linear surjective map $P:W\to V$. For every $\fS_n$-linear map $\varphi:\R^n\to\Sym_2(V)$ we then construct an $\Or(n)$-linear map $\Phi:\Sym_2(\R^n)\to\Sym_2(W)$ such that the diagram \[\begin{tikzcd}
   \R^n \arrow{r}{\varphi} \arrow{d}{\textrm{diag}} & \Sym_2(V) \\
   \Sym_2(\R^n) \arrow{r}{\Phi} & \Sym_2(W) \arrow{u}{\textrm{S}_2P}
  \end{tikzcd}\] commutes. Here $\textrm{diag}(a)$ denotes the diagonal matrix with diagonal $a\in\R^n$. We further show for all $a\in\R^n$ that $\Phi(\diag(a))$ is positive semidefinite if and only if $\varphi(a)=(\textrm{S}_2P)(\Phi(\diag(a)))$ is positive semidefinite. This implies \Cref{thm:main} since each real symmetric matrix can be diagonalized by an orthogonal transformation.

In \Cref{sec:deris} we will apply \Cref{thm:main} to the spectrahedral representation of elementary symmetric polynomials $\sigma_{d,n}$ from \cite{BranEle} and construct a spectrahedral representation of all derivative relaxations of the positive semidefinite cone. In \Cref{ex:firstderi} we note that applying \Cref{thm:main} to the spectrahedral description of $\sigma_{n-1,n}$ constructed in \cite{pol1} exactly gives us the construction from \cite{firstderi}.

Our very explicit approach makes it possible to deduce consequences for Wronskians and sums of squares in \Cref{sec:newton}.

\bigskip
\noindent \textbf{Acknowledgements.}
The question on spectrahedral representations of derivative relaxations of the positive semidefinite cone was posed at the second Problem Solving Day that took place at the Simons Institute for the Theory of Computing in the course of the program on the ``Geometry of Polynomials'' in spring 2019. I would like to thank Kuikui Liu, Claus Scheiderer, Nikhil Srivastava and especially Levent Tun\c{c}el for stimulating discussions during this Problem Solving Day and thereafter. Further I would like to thank Peter B\"urgisser and Philipp Reichenbach for pointing to literature regarding the representation theory of the orthogonal group and Petter Br\"and\'en for comments on the inequality in \Cref{cor:ineq1}. Finally, I thank the anonymous referee for many helpful comments that improved the quality of this manuscript.

\section{Some representation theory}\label{sec:preps}
For any natural number $n$ we let $[n]=\{1,\ldots,n\}$. For any set $S$ we denote by $\binom{S}{d}$ the set of $d$-element subsets of $S$. For all natural numbers $d,n$ with $0\leq d\leq n$ we consider the real vector space $\Ma_{d,n}$ of all multiaffine homogeneous polynomials of degree $d$ in $n$ variables, i.e., the subspace of $\R[x_1,\ldots,x_n]_d$ that is spanned by square-free monomials. For any subset $I\subset[n]=\{1,\ldots,n\}$ we let $\sigma_{d}(I)$ be the elementary symmetric polynomial of degree $d$ in the variables indexed by $I$. We always have $\sigma_{d}(I)\in\Ma_{d,n}$. We further denote by $\delta_i$ the $i$th unit vector in $\R^n$.

\subsection{Some representation theory of $\fS_n$}\label{sec:symn}
Let $\fS_n$ be the group of all permutations of $[n]$. Recall for example from \cite[\S4.1]{fultonharris} that irreducible representations of $\fS_n$ are in bijection to \emph{partitions} of $n$, i.e. tuples $\lambda=(\lambda_1,\ldots,\lambda_r)$ of positive integers $\lambda_i$ such that $\lambda_1\geq\cdots\geq\lambda_r$ and $n=\lambda_1+\cdots+\lambda_r$. The integer $r$ is called the \emph{length} of $\lambda$. Given a partition $\lambda=(\lambda_1,\ldots,\lambda_r)$ of $n$, the \emph{conjugate partition} $\lambda'$ is defined as $(\mu_1,\ldots,\mu_{\lambda_1})$ where $\mu_k$ denotes the number of indices $i\in[r]$ such that $\lambda_i\geq k$.
We denote the irreducible $\fS_n$-module corresponding to the partition $\lambda=(\lambda_1,\ldots,\lambda_r)$ of $n$ by $V_\lambda=V_{\lambda_1,\ldots,\lambda_r}$ as in \cite{fultonharris}. However, unlike in \cite{fultonharris}, we consider \emph{real} representations of $\fS_n$ rather than complex representations. Since each irreducible representation of $\fS_n$ can in fact be defined over the rational numbers \cite[p.~46]{fultonharris}, this will not cause any problems. It implies that on the real vector space $V_\lambda$ there is an invariant scalar product and the elements of $\fS_n$ act on $V_\lambda$ as orthogonal transformations. For representations $V$ and $W$ of $\fS_n$ we denote $$(V,W)_{\fS_n}=\dim(\Hom_{\fS_n}(V,W))$$ the dimension of all $\fS_n$-linear maps from $V$ to $W$.

Now consider the natural action of $\fS_n$ on $\Ma_{d,n}$ that is given by permuting the variables. There is a unique scalar product on the vector space $\Ma_{d,n}$ that has the monomials as orthonormal basis. Clearly, this scalar product is invariant under the action of $\fS_n$. We will always identify $\Ma_{d,n}$ with its dual representation via this scalar product. As a first step we decompose $\Ma_{d,n}$ into irreducible representations.

\begin{lem}\label{lem:decma}
 We have $\Ma_{d,n}\cong\oplus_{i=0}^{\min(d,n-d)} V_{n-i,i}$.
\end{lem}

\begin{proof}
 It is straightforward to see that $\Ma_{d,n}$ is the representation of $\fS_n$ induced by the trivial representation of $\fS_d\times\fS_{n-d}$. Then the claim follows directly from Young's rule \cite[Cor.~4.39]{fultonharris} as pointed out in \cite[p.~57]{fultonharris}.
\end{proof}

\begin{cor}\label{cor:mashort}
 $\Ma_{d,n}$ is a short representation of $\fS_n$.
\end{cor}

\begin{ex}\label{ex:rninma}
 \Cref{lem:decma} says in partiular that we can embed $\R^n\cong V_n\oplus V_{n-1,1}$ $\fS_n$-linearly to $\Ma_{d,n}$ if $0<d<n$. We claim that such an embedding is given by $$\iota_d:\R^n\to\Ma_{d,n},\, \delta_i\mapsto x_i\cdot \sigma_{d-1}([n]\setminus\{i\}).$$ Indeed, this map is clearly $\fS_n$-linear. In order to show that it is injective, it suffices to find one vector in each irreducible component of $\R^n$ that is not sent to zero. To this end note that the all-ones vector $e=\sum_{i=1}^n\delta_i\in V_n$ is mapped to $d\cdot\sigma_d([n])\neq0$. Further $\delta_1-\delta_2\in V_{n-1,1}$ gets sent to $x_1\cdot \sigma_{d-1}([n]\setminus\{1\})-x_2\cdot \sigma_{d-1}([n]\setminus\{2\})$. Assume that this is zero, i.e., that $$x_1\cdot \sigma_{d-1}([n]\setminus\{1\})=x_2\cdot \sigma_{d-1}([n]\setminus\{2\}).$$ This implies that $x_2$ divides $\sigma_{d-1}([n]\setminus\{1\})$ which is only possible if $d=n$.
 \demo \end{ex}

\begin{ex}\label{ex:ma14dec}
 In bases the decomposition $\Ma_{1,4}=V_4\oplus V_{3,1}$ is given by $$V_{3,1}=\textrm{Span}(x_1-x_2,x_1-x_3,x_1-x_4)$$ and its orthogonal complement $V_4$ spanned by $x_1+x_2+x_3+x_4$.
\demo \end{ex}

For any $a\in\R^n$ and $0\leq d<n$ we consider the map $$\De_a: \Ma_{d+1,n}\to \Ma_{d,n},\, f\mapsto \textrm{D}_af:=\sum_{i=1}^n a_i\frac{\partial f}{\partial x_i}.$$ The map $\De_e$ for $e=(1,\ldots,1)$ is clearly a homomorphism of $\fS_n$-modules. In what follows next, we study the properties of the map $\De_e$ more closely and show how it can be used to make the decomposition of $\Ma_{d,n}$ into irreducibles explicit.

\begin{figure}[h]
  \begin{center}
    \includegraphics[width=4cm]{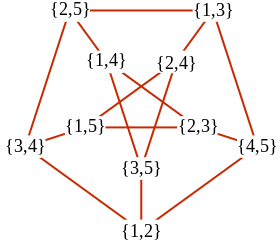}
  \end{center}
  \caption{The Kneser graph $K(5,2)$.}
  \label{fig:kneser}
\end{figure}

\begin{lem}\label{lem:kneser}
 If $2d\leq n$, then $\De^{n-2d}_e: \Ma_{n-d,n}\to \Ma_{d,n}$ is an isomorphism.
\end{lem}

\begin{proof}
 Consider the isomorphism $$\psi:\Ma_{d,n}\to\Ma_{n-d,n},\, \prod_{i\in T} x_i\mapsto\frac{1}{(n-2d)!}\prod_{i\not\in T} x_i \textrm{ for } T\in\binom{[n]}{d}.$$It suffices show that $\Psi=\psi\circ\De^{n-2d}_e$ is an isomorphism.
 For $S\in\binom{[n]}{n-d}$ we have  $$\De^{n-2d}_e \prod_{i\in S} x_i=(n-2d)!\cdot \sum_{T\in\binom{S}{d}}\prod_{i\in T} x_i.$$ Therefore, we have $$\Psi\left(\prod_{i\in S} x_i\right)=\sum_{T\in\binom{[n]}{n-d},\, S\cap T=\emptyset}\prod_{i\in T} x_i.$$
 So the representing matrix of $\Psi$ with respect to the monomial basis is the adjacency matrix of the \emph{Kneser graph} $K(n,n-d)$: This is the graph which has $\binom{[n]}{n-d}$ as its set of vertices, and two subsets of $[n]$ are adjacent if and only if they are disjoint. This matrix is known to have full rank, see e.g. \cite[Cor.~6.6.1]{knesergraph}.
\end{proof}

\begin{cor}\label{cor:fullrankma}
 Let $1\leq d\leq n$. The map $\De_e: \Ma_{d,n}\to \Ma_{d-1,n}$ has full rank: It is injective if $2d>n$ and surjective if $2d-2<n$.
\end{cor}

\begin{proof}
 We have that $$\dim(\Ma_{d,n})\leq\dim(\Ma_{d-1,n})\Leftrightarrow\binom{n}{d}\leq\binom{n}{d-1}\Leftrightarrow2d>n.$$ In this case we therefore have to show that $\De_e$ is injective. By \Cref{lem:kneser} the map $$\De_e^{2d-n-1}\circ\De_e:\Ma_{d,n}\to\Ma_{n-d,n}$$ is injective and thus is $\De_e:\Ma_{d,n}\to\Ma_{d-1,n}$. The other case follows analogously.
\end{proof}

\begin{cor}\label{cor:kern}
 Let $0\leq 2d\leq n$. The kernel of $\De_e: \Ma_{d,n}\to \Ma_{d-1,n}$ is isomorphic to the $\fS_n$-module $V_{n-d,d}$.
\end{cor}

\begin{proof}
 This is clear because $\De_e$ is surjective by \Cref{cor:fullrankma} and because $\Ma_{d,n}\cong\Ma_{d-1,n}\oplus V_{n-d,d}$ by \Cref{lem:decma}.
\end{proof}

\begin{cor}\label{cor:interscor}
 Let $0\leq2d<n$, and consider $V_{n-d,d}$ as a subset of $\Ma_{d,n}$ via the isomorphism from \Cref{lem:decma}. Then we have that $\R[x_1,\ldots,x_{n-1}]\cap V_{n-d,d}\neq\{0\}$.
\end{cor}

\begin{proof}
 The map $\De_e$ maps $\Ma_{d,n-1}$ to $\Ma_{d-1,n-1}$. Thus its kernel intersects $\Ma_{d,n-1}$ nontrivially for dimension reasons.
\end{proof}

\begin{ex}\label{ex:ma24dec}
By \Cref{lem:decma} we know that $$\Ma_{2,4}=V_4\oplus V_{3,1}\oplus V_{2,2}.$$
We want to compute this decomposition explicitly. By \Cref{cor:kern} the component $V_{2,2}$ is the kernel of $\De_e:\Ma_{2,4}\to\Ma_{1,4}$. Its representing matrix with respect to the monomial bases is given by:
\[ 
\bordermatrix{
	  &x_1x_2 &  x_1x_3 &  x_1x_4 & x_2x_3 & x_2x_4 & x_3x_4 \cr
x_1 & 1 & 1 & 1 & 0 & 0 & 0  \cr
x_2 & 1 & 0 & 0 & 1 & 1 & 0  \cr
x_3  & 0 & 1 & 0 & 1 & 0 & 1  \cr
x_4  & 0 & 0 & 1 & 0 & 1 & 1  \cr
}.
\] Its kernel and thus $V_{2,2}$ is spanned $(x_1-x_4)(x_2-x_3)$ and $(x_1-x_3)(x_2-x_4)$. The orthogonal complement of $V_{2,2}$ in $\Ma_{2,4}$ is $V_4\oplus V_{3,1}$ and can be computed as  \[W=\textrm{Span}(
x_1(x_2+x_3+x_4),x_2(x_1+x_3+x_4),x_3(x_1+x_2+x_4),x_4(x_1+x_2+x_3)).
\]Another application of \Cref{cor:kern} shows that $V_{3,1}$ is the kernel of $$\De_e^2: W\to\Ma_{0,4}=\R$$which is spanned by $(x_1-x_2)(x_3+x_4)$, $(x_1-x_3)(x_2+x_4)$ and $(x_1-x_4)(x_2+x_3)$. Finally, the invariant part $V_4$ is of course spanned by $$\sigma_{2,4}=x_1x_2+x_1x_3+x_1x_4+x_2x_3+x_2x_4+x_3x_4.$$\demo
 \end{ex}

By Schur's Lemma the multiplicity of the trivial representation $V_{n}$ in both $\Sym_2 V_{n-d,d}$ and $V_{n-d,d}\otimes V_{n-d,d}$ is $1$ for $0\leq2d\leq n$. We now compute the multiplicity of $V_{n-1,1}$ in these representations of $\fS_n$. 

\begin{lem}\label{lem:stdinsym}
Let $0\leq2d\leq n$. The multiplicity of $V_{n-1,1}$ in both $\Sym_2 V_{n-d,d}$ and $V_{n-d,d}\otimes V_{n-d,d}$ is $1$ if $0<2d<n$ and $0$ otherwise.
\end{lem}

\begin{proof}
We consider the usual inclusion of $\fS_{n-1}$ in $\fS_n$. Then we have $$(V_{n-1,1},\Sym_2 V_{n-d,d})_{\fS_n}=(V_{n-1,1}\oplus V_{n},\Sym_2 V_{n-d,d})_{\fS_n}-1$$since the multiplicity of the trivial representation $V_{n}$ in $\Sym_2 V_{n-d,d}$ is $1$ \cite[Ex.~4.5.1b)]{fultonharris}. By Frobenius Reciprocity \cite[Cor.~3.20]{fultonharris} and because we have that $V_{n-1,1}\oplus V_{n}=\textrm{Ind}^{\fS_n}_{\fS_{n-1}} V_{n-1}$ it follows that$$(V_{n-1,1},\Sym_2 V_{n-d,d})_{\fS_n}=(V_{n-1},\Sym_2( \textrm{Res}^{\fS_n}_{\fS_{n-1}} V_{n-d,d}))_{\fS_{n-1}}-1 .$$ By Pieri's Rule \cite[Ex.~4.44]{fultonharris} we have that  
$\textrm{Res}^{\fS_n}_{\fS_{n-1}} V_{n-d,d}=V_{n-d-1,d}\oplus V_{n-d,d-1}$ if $0<2d<n$. Otherwise, there is only one summand. Using \cite[Ex.~4.5.1b)]{fultonharris} again implies then the claim for $\Sym_2 V_{n-d,d}$. The proof for $V_{n-d,d}\otimes V_{n-d,d}$ is verbatim the same after replacing $\Sym_2 V_{n-d,d}$ by $V_{n-d,d}\otimes V_{n-d,d}$.
\end{proof}

For describing the components isomorphic to $V_{n}$ and $V_{n-1,1}$ in $\Sym_2 (V_{n-d,d})$, we consider the diagonal map $\diag:\Ma_{d,n}\to\Sym_2(\Ma_{d,n})$ that sends a monomial $m$ to $m\otimes m$. This map is clearly $\fS_n$-invariant. By restricting this map to $V_{n}$ resp. $V_{n-1,1}$ and projecting to $V_{n-d,d}\subset\Ma_{d,n}$, we obtain $\fS_n$-invariant maps $$\alpha_{d,n}:V_{n}\to\Sym_2(V_{n-d,d})\subset V_{n-d,d}\otimes V_{n-d,d}$$ and $$\beta_{d,n}:V_{n-1,1}\to\Sym_2(V_{n-d,d})\subset V_{n-d,d}\otimes V_{n-d,d}.$$ The next lemmas show that both maps are nonzero.

\begin{lem}\label{lem:trvsym}
 Let $0<2d<n$. The map $\alpha_{d,n}:V_{n}\to\Sym_2(V_{n-d,d})$ is nonzero.
\end{lem}

\begin{proof}
 The invariant part of $\Ma_{d,n}$ is spanned by $\sigma_{d}([n])$ which is mapped by $\diag$ to the identity matrix. This is positive definite and so is its restriction to $V_{n-d,d}$ which is in particular nontrivial.
\end{proof}

\begin{lem}\label{lem:symsym}
 Let $0<2d<n$. The map $\beta_{d,n}:V_{n-1,1}\to\Sym_2(V_{n-d,d})$ is nonzero.
\end{lem}

\begin{proof}
 Consider the $\fS_n$-linear map $\R^n\to\Ma_{d,n}$ that sends the $i$th unit vector to $x_i\cdot\sigma_{d-1}([n]\setminus\{i\})$, see \Cref{ex:rninma}. The vector $e-n\cdot \delta_n$, where $e$ is the all-ones vector and $\delta_n$ the $n$th unit vector, lies in the $V_{n-1,1}$-part of $\R^n$. It is sent to $$d \sigma_d([n])-nx_n\sigma_{d-1}([n-1])= d \sigma_{d}([n-1])+(d-n)x_n \sigma_{d-1}([n-1]).$$ This element gets mapped by the map $\diag$ to a diagonal matrix all whose diagonal entries are $d$ or $d-n$ according to whether $x_n$ occurs in the corresponding monomial or not.  The restriction of the corresponding bilinear form to $\Ma_{d,n-1}$ is thus positive definite. Since $2d<n$ we have that $\Ma_{d,n-1}\cap V_{n-d,d}\neq\{0\}$ by \Cref{cor:interscor} which implies that the restriction of this bilinear form to $V_{n-d,d}$ is nontrivial.
\end{proof}

\begin{ex}\label{ex:n41}
 We describe the components isomorphic to $V_4$ and $V_{3,1}$ in the $\fS_4$-module $\Sym_2(V_{4-d,d})$ for $d=0,1,2$ explicitly by means of a basis.
 \begin{enumerate}[a)]
  \item We have $\Sym_2(V_4)=V_4$. A basis $V_4$ is given by any nonzero bilinear form.
  \item We have $\Sym_2(V_{2,2})=V_4\oplus V_{2,2}$. A basis of $V_4$ is given by the symmetric bilinear form whose Gram matrix with respect to the basis calculated in \Cref{ex:ma24dec} is:
  \[ 
\bordermatrix{
	  &(x_1-x_4)(x_2-x_3) &  (x_1-x_3)(x_2-x_4)  \cr
(x_1-x_4)(x_2-x_3) & 2 & 1   \cr
(x_1-x_3)(x_2-x_4) & 1 & 2   \cr
}.
\]
\item Finally, we have $\Sym_2(V_{3,1})=V_4\oplus V_{3,1}\oplus V_{2,2}$. For any $a\in\R^4$ we consider the symmetric bilinear form $G(a)$ whose Gram matrix with respect to the basis of $V_{3,1}$ calculated in \Cref{ex:ma14dec} is:
  \[
   \bordermatrix{
	  &x_1-x_2 &  x_1-x_3 & x_1-x_4  \cr
x_1-x_2 & a_1+a_2 & a_1 & a_1   \cr
x_1-x_3 & a_1 & a_1+a_3 & a_1   \cr
x_1-x_4 & a_1 & a_1 & a_1+a_4   \cr
}.
  \]
Restricting the map $a\mapsto G(a)$ to $V_4\subset\R^4$ and $V_{3,1}\subset\R^4$ respectively, we obtain the maps $\alpha_{1,4}$ resp. $\beta_{1,4}$.\demo
 \end{enumerate}
 \end{ex}
 
By Schur's Lemma the multiplicity of the trivial representation $V_{n}$ in $V_\lambda\otimes V_\mu$ is zero when $\lambda\neq\mu$. We now compute the multiplicity of $V_{n-1,1}$ in these representations of $\fS_n$ for short $\lambda$ and $\mu$.

\begin{lem}\label{lem:stdintens}
 Let $0\leq 2d<2d'\leq n$. The multiplicity of $V_{n-1,1}$ in $V_{n-d,d}\otimes V_{n-d',d'}$ is $1$ if $d'=d+1$ and $0$ otherwise.
\end{lem}

\begin{proof}
 We use similar arguments as in \Cref{lem:stdinsym} to compute this multiplicity. The assumption $0\leq 2d<2d'\leq n$ implies that $V_{n-d,d}$ and $V_{n-d',d'}$ are nonisomorphic irreducible representations of $\fS_n$. Thus $V_n$ does not appear in $V_{n-d,d}\otimes V_{n-d',d'}$ and we have $$(V_{n-1,1},V_{n-d,d}\otimes V_{n-d',d'})_{\fS_n}=(V_n\oplus V_{n-1,1},V_{n-d,d}\otimes V_{n-d',d'})_{\fS_n}.$$ By Frobenius Reciprocity \cite[Cor.~3.20]{fultonharris} and because we have that $V_{n-1,1}\oplus V_{n}=\textrm{Ind}^{\fS_n}_{\fS_{n-1}} V_{n-1}$ it follows that
 $$ (V_{n-1,1},V_{n-d,d}\otimes V_{n-d',d'})_{\fS_n}=(V_{n-1},\textrm{Res}^{\fS_n}_{\fS_{n-1}}V_{n-d,d}\otimes \textrm{Res}^{\fS_n}_{\fS_{n-1}}V_{n-d',d'})_{\fS_{n-1}}. $$ By Pieri's Rule \cite[Ex.~4.44]{fultonharris} we find that the only possibility for $\textrm{Res}^{\fS_n}_{\fS_{n-1}}V_{n-d,d}$ and $\textrm{Res}^{\fS_n}_{\fS_{n-1}}V_{n-d',d'}$ to share an irreducible component is that $d'=d+1$ in which case we have $(V_{n-1,1},V_{n-d,d}\otimes V_{n-d',d'})_{\fS_n}=1$.
\end{proof}

\begin{rem}\label{rem:symtens}
Let $0\leq 2d\leq n-2$. We can explicitly describe the component isomorphic to $V_{n-1,1}$ in $V_{n-d,d}\otimes V_{n-d-1,d+1}$. Consider the $\fS_n$-linear map $$\R^n\to\Hom(\Ma_{d+1,n}, \Ma_{d,n}),\, a\mapsto \De_a.$$Restricting $\De_a$ to the kernel of $\De_e$, we get an $\fS_n$-linear map $$\R^n\to\Hom(V_{n-d-1,d+1}, V_{n-d,d})\cong V_{n-d,d}\otimes V_{n-d-1,d+1}$$ since $\De_e\De_af=\De_a\De_ef=0$ for all $f$ in the kernel of $\De_e$ and by \Cref{cor:kern}.
This map is nonzero because for each homogeneous polynomial of positive degree at least one directional derivative is nonzero.
The restriction to $V_n\subset\R^n$ is zero. Therefore, the restriction to $V_{n-1,1}$ cannot be zero as well and thus gives us the desired embedding $\gamma_{d,n}: V_{n-1,1}\to V_{n-d,d}\otimes V_{n-d-1,d+1}$. For any $a\in V_{n-1,1}\subset\R^n$ the bilinear form $\gamma_{d,n}(a)$ sends a pair $(f,g)\in V_{n-d,d}\times V_{n-d-1,d+1}\subset\Ma_{d,n}\times\Ma_{d+1,n}$ to the scalar product $\langle f,\De_a g\rangle$.
\end{rem}

The maps $\alpha_{d,n},\beta_{d,n}$ and $\gamma_{d,n}$ allow us to completely describe the vector space of $\fS_n$-linear maps $\R^n\to\Sym_2(V)$ for any short representation $V$ of $\fS_n$. The next two examples illustrate this for $V=\Ma_{2,4}$.

\begin{ex}\label{ex:n42}
 We describe the component isomorphic to $V_{3,1}$ in the $\fS_4$-module $V_{4-d,d}\otimes V_{3-d,d+1}$ for $d=0,1$ explicitly by means of a basis.
 \begin{enumerate}[a)]
  \item We have $V_{4}\otimes V_{3,1}=V_{3,1}$. For any $a\in\R^4$ we consider the map $V_{3,1}\to V_4=\R$ whose representing matrix with respect to the basis of $V_{3,1}$ calculated in \Cref{ex:ma14dec} is:
  \[ 
\bordermatrix{
	  &x_1-x_2 &  x_1-x_3 & x_1-x_4  \cr
1 & a_1-a_2 & a_1-a_3 & a_1-a_4   \cr
}.
\]Restricting this map to $a\in V_{3,1}\subset\R^4$, we obtain the map $\gamma_{0,4}$.
  \item We have $V_{3,1}\otimes V_{2,2}=V_{3,1}\oplus V_{2,1,1}$. For any $a\in\R^4$ we consider the map $V_{2,2}\to V_{3,1}$ whose representing matrix with respect to the basis of $V_{2,2}$ calculated in \Cref{ex:ma24dec} and the dual basis of the one calculated in \Cref{ex:ma14dec} is:
  \[ 
\bordermatrix{
	  &(x_1-x_4)(x_2-x_3) &  (x_1-x_3)(x_2-x_4)  \cr
\frac{1}{4}(x_1-3x_2+x_3+x_4) & a_4-a_1+a_2-a_3  & a_3-a_1 +a_2-a_4  \cr
\frac{1}{4}(x_1+x_2-3x_3+x_4) & a_1-a_4+a_2-a_3 &  2a_2-2a_4  \cr
\frac{1}{4}(x_1+x_2+x_3-3x_4) &  2a_2-2a_3 &  a_1-a_3 +a_2-a_4 \cr
}.
\]Restricting this map to $a\in V_{3,1}\subset\R^4$, we obtain the map $\gamma_{1,4}$.
 \end{enumerate}
 \end{ex}

\begin{figure}[h]
  \begin{center}
    \includegraphics[width=4cm]{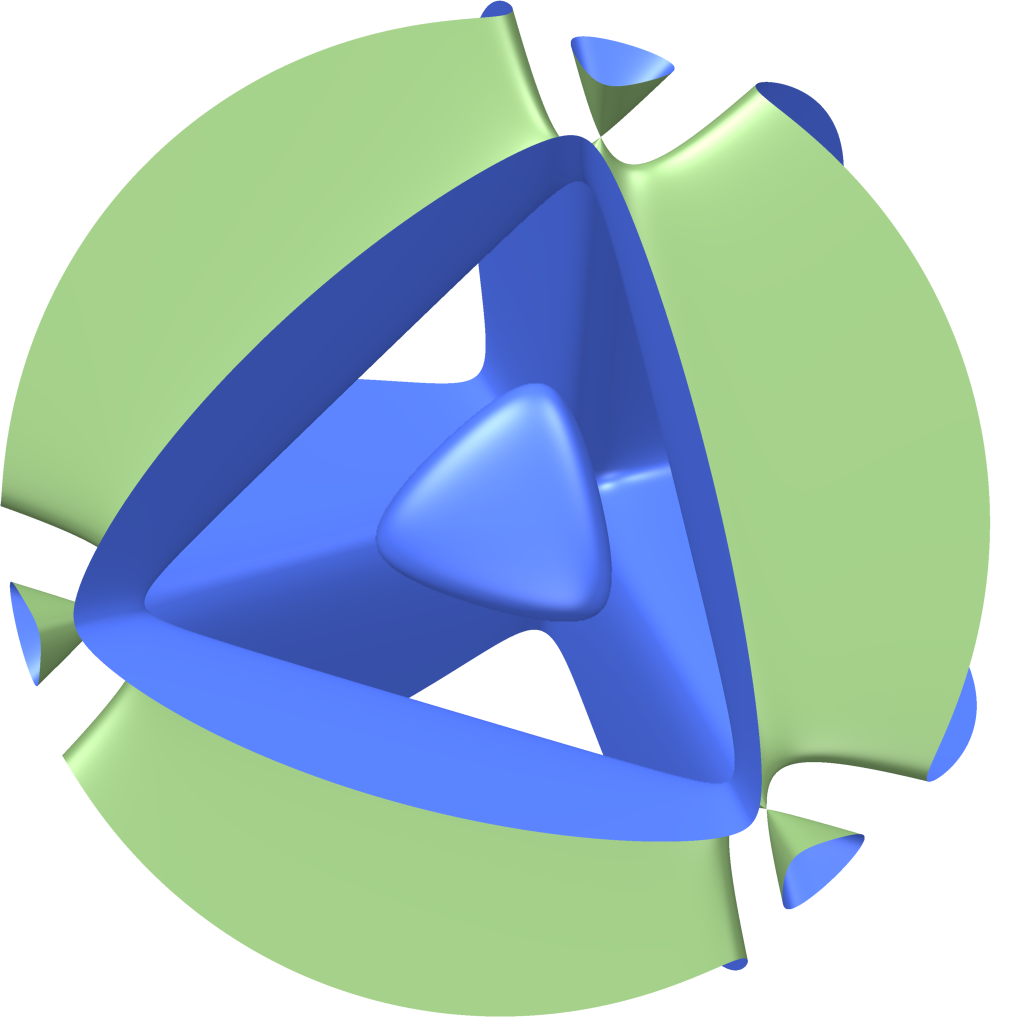}
  \end{center}
  \caption{The points in the affine hyperplane $A=1$ where the matrix from \Cref{ex:ma24complete} has rank $<4$ for $\lambda_1=\cdots=\lambda_4=1$ and $\lambda_5=\lambda_6=0$. The convex region in the middle is the spectrahedron.}
  \label{fig:quartic}
\end{figure}

\begin{ex}\label{ex:ma24complete}
 Combining \Cref{ex:n41} and \Cref{ex:n42} we can completely describe all $\fS_4$-linear maps $\R^4\to\Sym_2(\Ma_{2,4})$. To that end let $M_1(a)$ be the $1\times 1$ matrix with entry $A:=a_1+a_2+a_3+a_4$. Further let $M_5(a)$ be the matrix from \Cref{ex:n41}b) multiplied by $A$, let $M_3(a)$ be the matrix from \Cref{ex:n41}c) and $M_4(a)=A\cdot M_3(1,1,1,1)$. Finally, let $M_2(a)$ and $M_6(a)$ be the matrices from \Cref{ex:n42}a) and \Cref{ex:n42}b) respectively. Then every $\fS_4$-linear map $\R^4\to\Sym_2(\Ma_{2,4})$ is of the form $$\bordermatrix{
	  &V_4 &  V_{3,1} & V_{2,2}  \cr
V_4 & \lambda_1 M_1(a)  & \lambda_2 M_2(a) & 0  \cr
V_{3,1} & \lambda_2 M_2(a)^t &  \lambda_3M_3(a)+\lambda_4M_4(a) & \lambda_6 M_6(a)  \cr
V_{2,2} &  0 &  \lambda_6 M_6(a)^t & \lambda_5M_5(a) \cr
}
 $$for some $\lambda_1,\ldots,\lambda_6\in\R$.
\demo \end{ex}

Having a basis of the vector space of $\fS_n$-linear maps $\R^n\to\Sym_2(V)$ for any short representation $V$ of $\fS_n$, we next want to make an analogous construction for certain representations of $\Or(n)$.

\subsection{Some representation theory of $\Or(n)$}\label{sec:on}
We consider the standard scalar product on $\R^n$ for which the unit vectors form an orthonormal basis: $$\langle x,y\rangle=\sum_{i=1}^nx_iy_i.$$ The \emph{orthogonal group} $\Or(n)$ is the group of all invertible linear maps $\R^n\to\R^n$ that are orthogonal with respect to this scalar product. In particular, the vector space $\R^n$ is a representation of $\Or(n)$ which is isomorphic to the dual representation $(\R^n)^*$. Recall that the scalar product on $\R^n$ induces a scalar product on $\wedge^d\R^n$: $$\langle v_1\wedge\cdots\wedge v_d, w_1\wedge\cdots\wedge w_d\rangle=\det(\langle v_i, w_j\rangle)_{1\leq i,j\leq d}.$$Clearly this inner product is invariant under $\Or(n)$ and thus yields an isomorphism $\wedge^d\R^n\cong(\wedge^d\R^n)^*$ of $\Or(n)$-modules. For any $S=\{s_1,\ldots,s_d\}\subset[n]$ with $s_1<\cdots<s_d$ we denote by $e_S\in\wedge^d\R^n$ the element $\delta_{s_1}\wedge\cdots\wedge \delta_{s_d}$ where $\delta_i$ is the $i$th unit vector. The elements $e_S$ form an orthonormal basis with respect to the above scalar product. We further have a nondegenerate pairing $$\wedge^d\R^n\times(\wedge^{n-d}\R^n\otimes\wedge^n\R^n)\to\R,\, (\alpha,\beta\otimes\gamma)\mapsto\langle\alpha\wedge\beta,\gamma\rangle$$which is also $\Or(n)$-invariant and thus gives an isomorphism $\wedge^d\R^n\cong\wedge^{n-d}\R^n\otimes\wedge^n\R^n$ of $\Or(n)$-modules. On our basis this isomorphism operates in the following way: $$e_S\mapsto e_{S^c}\otimes (e_S\wedge e_{S^c}).$$
This also shows that $\wedge^{n}\R^n\otimes\wedge^n\R^n\cong\wedge^{0}\R^n$, the trivial representation.

In order to prove \Cref{thm:main} we want to associate to each short representation of $\fS_n$ a suitable representation of the orthogonal group $\Or(n)$. The analog to $\Ma_{d,n}$ will be the $\Or(n)$-module $\Sym_2(\wedge^d\R^n)$ (for now). Via the scalar product that we defined above on $\wedge^d\R^n$, we can also consider elements $A,B\in\Sym_2(\wedge^d\R^n)$ as selfadjoint endomorphisms of $\wedge^d\R^n$. Thus we can define a scalar product on $\Sym_2(\wedge^d\R^n)$:$$(A,B)\mapsto\tr(A\cdot B).$$

We have to set up some notation. 
Let $I,J_1,J_2\subset[n]$ be pairwise disjoint subsets such that $|J_1|=|J_2|$ and $|I|+|J_1|=d$. We write $J_1\leq J_2$ if for all $j_2\in J_2$ there is a $j_1\in J_1$ such that $j_1\leq j_2$\footnote{There is nothing special about this particular ordering and any other would work as well.}. For $J_1\leq J_2$ we define the elements $t_{I,J_1,J_2}\in\Sym_2(\wedge^d\R^n)$ as follows:
$$t_{I,J_1,J_2}=\begin{cases}
                 e_I\otimes e_I, \textrm{ if }J_1=J_2=\emptyset,\\
                 \frac{\sqrt{2}}{2}(e_{I\cup J_1}\otimes e_{I\cup J_2}+e_{I\cup J_2}\otimes e_{I\cup J_1}), \textrm{ otherwise.}
                \end{cases} $$
We note that the set $$\{t_{I,J_1,J_2}: I,J_1,J_2\subset[n]\textrm{ p.w. disjoint s.t. }|J_1|=|J_2|,\,|I|+|J_1|=d  \textrm{ and }J_1\leq J_2\}$$ is an orthonormal basis of $\Sym_2(\wedge^d\R^n)$. We observe that the map \begin{align}\label{eqn:mainsym}\Ma_{d,n}\to\Sym_2(\wedge^d\R^n),\,\prod_{i\in I}x_i\mapsto e_I\otimes e_I \textrm{ for }I\in\binom{[n]}{d}\end{align} is an $\fS_n$-linear embedding. Here we consider $\Sym_2(\wedge^d\R^n)$ as an $\fS_n$-module via the natural inclusion $\fS_n\subset\Or(n)$. Like this we will always consider $\Ma_{d,n}$ as an $\fS_n$-invariant subspace of $\Sym_2(\wedge^d\R^n)$. Next we define an analog to the derivative $\De_a:\Ma_{d,n}\to\Ma_{d-1,n}$. 

\begin{cns}\label{cns:delta}
 For each $v\in\R^n$ we have a map $$\varphi_v:\wedge^{n-d}\R^n\to\wedge^{n-d+1}\R^n,\,\omega\mapsto \omega\wedge v.$$ Employing the isomorphism of $\Or(n)$-modules $\wedge^{n-i}\R^n\cong\wedge^{i}\R^n\otimes\wedge^n\R^n$ we obtain
 $$\wedge^{d}\R^n\otimes\wedge^n\R^n\to\wedge^{d-1}\R^n\otimes\wedge^n\R^n.$$Taking the tensor product with $\wedge^n\R^n$ we obtain the map $\psi_v:\wedge^{d}\R^n\to\wedge^{d-1}\R^n$. The map $\R^n\to\Hom(\wedge^{d}\R^n,\wedge^{d-1}\R^n)$ that sends $v$ to $\psi_v$ is $\Or(n)$-linear. The same is true for the induced map $$\R^n\otimes\R^n\to\Hom(\wedge^{d}\R^n\otimes\wedge^{d}\R^n, \wedge^{d-1}\R^n\otimes\wedge^{d-1}\R^n),\, v\otimes w\mapsto \psi_v\otimes\psi_w.$$
 
\begin{lem}
 Any symmetric tensor $\omega\in\R^n\otimes\R^n$ is sent to a homomorphism that maps $\Sym_2(\wedge^{d}\R^n)\subset \wedge^{d}\R^n\otimes\wedge^{d}\R^n$ to $\Sym_2(\wedge^{d-1}\R^n)\subset \wedge^{d-1}\R^n\otimes\wedge^{d-1}\R^n$.
\end{lem}

\begin{proof}
 It suffices to show the claim for $\omega=v\otimes v$ for $v\in\R^n$ as every element of $\Sym_2(\R^n)$ is a linear combination of such. If $\alpha\in\wedge^{d}\R^n$, then clearly $$(\psi_v\otimes\psi_v)(\alpha\otimes\alpha)=\psi_v(\alpha)\otimes\psi_v(\alpha)$$ is symmetric which shows the claim.
\end{proof}

Therefore, we obtain an $\Or(n)$-linear map $$\Sym_2(\R^n)\to\Hom(\Sym_2(\wedge^{d}\R^n),\Sym_2(\wedge^{d-1}\R^n)),\, A\mapsto \Delta_A.$$\demo
\end{cns}

The next compatibility lemma justifies that $\Delta_A$ can indeed be regarded as an analog to the derivative $\De_a$.

\begin{lem}\label{lem:derisagree}
 If $A\in\Sym_2(\R^n)$ is the diagonal matrix with diagonal $a\in\R^n$, then $\Ma_{d,n}\subset\Sym_2(\wedge^d\R^n)$ gets mapped by $\Delta_A$ to $\Ma_{d-1,n}\subset\Sym_2(\wedge^{d-1}\R^n)$ and the restriction of $\Delta_A$ to $\Ma_{d,n}$ is the derivative $\De_a:\Ma_{d,n}\to\Ma_{d-1,n}$. 
\end{lem}

\begin{proof}
 The image of $e_S$ under the map $\omega\mapsto\omega\wedge \delta_i$ is, up to a sign, $e_{S\cup\{i\}}$ if $i\not\in S$ and $0$ otherwise. Thus $\psi_{\delta_i}(e_T)$ is, again up to a sign, $e_{T\setminus\{i\}}$ if $i\in T$ and $0$ otherwise. Letting $E_{ii}$ be the diagonal matrix with diagonal $\delta_i$ we therefore have that $$\Delta_{E_{ii}}(e_T\otimes e_T)=\psi_{\delta_i}(e_T)\otimes\psi_{\delta_i}(e_T)=e_{T\setminus\{i\}}\otimes e_{T\setminus\{i\}}$$if $i\in T$ and $0$ otherwise. This shows the claim.
\end{proof}

We now decompose the $\Or(n)$-module $\Sym_2(\wedge^d\R^n)$ by means of the map $\Delta_I$ in the same manner that we have decomposed $\Ma_{d,n}$ using the map $\De_e$. For this we need that the maps $\Delta_A$ and $\Delta_B$ commute.

\begin{lem}\label{lem:deltacommutes}
 For every $A,B\in\Sym_2(\R^n)$ we have $\Delta_A\circ\Delta_B=\Delta_B\circ\Delta_A$.
\end{lem}

\begin{proof}
 We use the notation from \Cref{cns:delta}. For any $v,w\in\R^n$ we clearly have $\varphi_v\circ\varphi_w=-\varphi_w\circ\varphi_v$. Thus by construction we also have $\psi_v\circ\psi_w=-\psi_w\circ\psi_v$. It follows that for any $v_1,v_2,w_1,w_2\in\R^n$ we have $$(\psi_{v_1}\otimes\psi_{v_2})\circ(\psi_{w_1}\otimes\psi_{w_2})=(\psi_{v_1}\circ\psi_{w_1})\otimes(\psi_{v_2}\circ\psi_{w_2})$$  $$=(-\psi_{w_1}\circ\psi_{v_1})\otimes(-\psi_{w_2}\circ\psi_{v_2})=(\psi_{w_1}\otimes\psi_{w_2})\circ(\psi_{v_1}\otimes\psi_{v_2})$$which implies the claim.
\end{proof}

Let $I\in\Sym_2(\R^n)$ be the identity matrix. Then $$\Delta_I: \Sym_2(\wedge^d\R^n)\to \Sym_2(\wedge^{d-1}\R^n)$$ is $\Or(n)$-linear because $I$ is fixed under the action of $\Or(n)$. We denote its kernel by $W_{n-d,d}$. If $2d\leq n$, then the intersection of $W_{n-d,d}$ with $\Ma_{d,n}$ is $V_{n-d,d}$ by \Cref{cor:kern} and \Cref{lem:derisagree}.

\begin{ex}\label{ex:trace}
 In the case $d=1$ the above map $$\Delta_I:\Sym_2(\R^n)\to\R$$ is just given by the trace.
\demo \end{ex}

\begin{ex}
 Clearly $W_{n}$ is the trivial $\Or(n)$-module. Further the $\Or(n)$-module $W_{n-1,1}\subset\Sym_2(\R^n)$ is the space of traceless matrices by \Cref{ex:trace}. The decomposition of $\Sym_2(\R^n)$ into irreducible $\Or(n)$-modules is thus $W_n\oplus W_{n-1,1}$. The subspace of $\Sym_2(\R^n)$ that we identified with $\Ma_{1,n}$ is the space of diagonal matrices. Note that in general $W_{n-d,d}$ does not need to be irreducible for $d\geq2$.
\demo \end{ex}

Now we define analogs for the maps $\alpha_{d,n},\beta_{d,n}$ and $\gamma_{d,n}$. To this end note that every $X\in\Sym_2(\wedge^d\R^n)$ gives rise to a symmetric bilinear form $$b_X: \Sym_2(\wedge^d\R^n)\otimes\Sym_2(\wedge^d\R^n)\to\R,\, A\otimes B \mapsto \tr(AXB)$$and clearly the map $X\mapsto b_X$ is $\Or(n)$-invariant. The natural map $\GL(\R^n)\to\GL(\wedge^d\R^n)$ that sends an invertible endomorphism $X$ to the induced endomorphism $\wedge^dX$ is a homomorphism of Lie groups and thus induces a homomorphism $\End(\R^n)\to\End(\wedge^d\R^n)$ of Lie algebras that commutes with taking the adjoint of an endomorphism. Thus we get a linear map $\cL_d:\Sym_2(\R^n)\to\Sym_2(\wedge^d\R^n)$ which is even $\Or(n)$-linear. The matrix $\cL_d(X)$ is called the \emph{$d$th additive compound matrix} of $X\in\Sym_2(\R^n)$. See \cite[Thm.~2]{additivecompound} for a proof of the above mentioned and further properties. So we get an $\Or(n)$-linear map $$\Sym_2(\R^n)\to\Sym_2(\Sym_2(\wedge^d\R^n)), X\mapsto b_{\cL_d(X)}.$$

\begin{rem}
 The linear map $\cL_d(X)$ can also be defined by the rule
 $$\cL_d(X)(v_{i_1}\wedge\cdots\wedge v_{i_d})=\sum_{j=1}^dv_{i_1}\wedge\cdots\wedge v_{i_{j-1}} \wedge Xv_{i_j}\wedge v_{i_{j+1}}\cdots\wedge v_{i_d}. $$
\end{rem}

\begin{lem}\label{lem:orth1}
 Let $D\in\Sym_2(\R^n)$ be a diagonal matrix. Let $I^i,J^i_1,J^i_2\subset[n]$, $i=1,2$, be two different triples of pairwise disjoint subsets such that $|J^i_1|=|J^i_2|$, $|I^i|+|J^i_1|=d$ and $J^i_1\leq J^i_2$ for $i=1,2$. Then $b_{\cL_d(D)}(t_{I^1,J^1_1,J^1_2},t_{I^2,J_1^2,J_2^2})=0$.
\end{lem}

\begin{proof}
 The representing matrix of $\cL_d(D)$ with respect to the basis given by the $e_S$ for $S\in\binom{[n]}{d}$ is diagonal. Thus by construction of $b_{\cL_d(D)}$ the elements $t_{I^1,J^1_1,J^1_2}$ and $t_{I^2,J_1^2,J_2^2}$ are orthogonal with respect to $b_{\cL_d(D)}$.
\end{proof}

Let $2d\leq n$. Choosing $X$ from the subspace $W_n$ resp. $W_{n-1,1}$ of $\Sym_2(\R^n)$ and restricting the bilinear form $b_{\cL_d(X)}$ to $W_{n-d,d}\subset\Sym_2(\wedge^d\R^n)$ we obtain $\Or(n)$-linear maps $A_{d,n}: W_n\to\Sym_2(W_{n-d,d})$ and $B_{d,n}:W_{n-1,1}\to\Sym_2(W_{n-d,d})$ respectively.

\begin{lem}\label{lem:resab}
 Let $\Lambda\in\Sym_2(\R^n)$ be the diagonal matrix with diagonal $\lambda\in\R^n$. If $\lambda\in V_n$, then the restriction of the bilinear form $A_{d,n}(\Lambda)$ to $V_{n-d,d}$ is $\alpha_{d,n}(\lambda)$. If $\lambda\in V_{n-1,1}$, then the restriction of the bilinear form $B_{d,n}(\Lambda)$ to $V_{n-d,d}$ is $\beta_{d,n}(\lambda)$.
\end{lem}

\begin{proof}
 The additive compound matrix $\cL_d(\Lambda)$ is also a diagonal matrix. We first restrict the bilinear form $b_{\cL_d(\Lambda)}$ to the space of diagonal matrices in $\Sym_2(\wedge^d(\R^n))$ which we have identified with $\Ma_{d,n}$. We denote the bilinear form on $\Ma_{d,n}$ that we get in this way by $B_\lambda$.
 
 First consider the case when $\lambda=\delta_i$ is a unit vector. In that case the diagonal entry of $b_{\cL_d(\Lambda)}$ corresponding to $e_S$ is $1$ if $i\in S$ and $0$ otherwise. Further for any $S,S'\in\binom{[n]}{d}$ we have that $B_{\delta_i}(\prod_{i\in S}x_i,\prod_{i\in S'}x_i)$ is $1$ if $i\in S$ and $S=S'$, and $0$ otherwise. Recall that the map $\diag:\Ma_{d,n}\to\Sym_2(\Ma_{d,n})$ sends a monomial $m$ to $m\otimes m$. Therefore, we have that $$B_{\delta_i}=\diag(x_i\cdot\sigma_{d-1}([n]\setminus\{i\}))=\diag(\iota_d(\delta_i))$$where $\iota_d:\R^n\to\Ma_{d,n}$ is the map considered in \Cref{ex:rninma}.
 Thus for arbitrary $\lambda\in\R^n$ we have $B_\lambda=\diag(\iota_d(\lambda))$. For $\lambda\in V_n$ or $\lambda\in V_{n-1,1}$ the restriction of $\diag(\iota_d(\lambda))$ to $V_{n-d,d}$ is exactly the definition of $\alpha_{d,n}(\lambda)$ and $\beta_{d,n}(\lambda)$ respectively.
\end{proof}

Finally, we note that for every $A\in\Sym_2(\R^n)$ the map $$\Delta_A:\Sym_2(\wedge^{d+1}\R^n)\to\Sym_2(\wedge^{d}\R^n)$$ sends $W_{n-d-1,d+1}$ to $W_{n-d,d}$ because $\Delta_A\circ\Delta_I=\Delta_I\circ\Delta_A$ by \Cref{lem:deltacommutes}. This gives an $\Or(n)$-linear map $$\Sym_2(\R^n)\to\Hom(W_{n-d-1,d+1},W_{n-d,d})=W_{n-d,d}\otimes W_{n-d-1,d+1}.$$Restricting that to $W_{n-1,1}$ yields the $\Or(n)$-linear map $$C_{d,n}:W_{n-1,1}\to W_{n-d,d}\otimes W_{n-d-1,d+1}.$$
More precisely, for a traceless symmetric matrix $A\in W_{n-1,1}\subset\Sym_2(\R^n)$, $F\in W_{n-d,d}\subset\Sym_2(\wedge^{d}\R^n)$ and $G\in W_{n-d-1,d+1}\subset\Sym_2(\wedge^{d+1}\R^n)$ the bilinear form $C_{d,n}(A)$ sends $(F,G)$ to the scalar product $\langle F, \Delta_A G \rangle$.

\begin{lem}\label{lem:resc}
 Let $\Lambda\in\Sym_2(\R^n)$ be the diagonal matrix with diagonal $\lambda\in V_{n-1,1}\subset\R^n$. Then the restriction of $C_{d,n}(\Lambda)$ to $V_{n-d,d}\otimes V_{n-d-1,d+1}$ is $\gamma_{d,n}(\lambda)$.
\end{lem}

\begin{proof}
 This follows directly from \Cref{lem:derisagree} and the definition of $\gamma_{d,n}(\lambda)$.
\end{proof}

Let us draw up an informal interim balance. We have seen in \Cref{sec:symn} that the maps $\alpha_{d,n},\beta_{d,n}$ and $\gamma_{d,n}$ serve as building blocks for $\fS_n$-linear maps $\varphi:\R^n\to\Sym_2(V)$ where $V$ is a short representation. By replacing each copy of $V_{n-d,d}$ in $V$ by a copy of $W_{n-d,d}$ we can associate to $V$ an $\Or(n)$-module $W$. Further by replacing the maps of type $\alpha_{d,n},\beta_{d,n}$ and $\gamma_{d,n}$ that constitute $\varphi$ with the corresponding maps $A_{d,n}$, $B_{d,n}$ and $C_{d,n}$ we obtain a map $\Phi:\Sym_2(\R^n)\to\Sym_2(W)$. \Cref{lem:resab} and \Cref{lem:resc} then imply that the diagram \[\begin{tikzcd}
   \R^n \arrow{r}{\varphi} \arrow{d}{\textrm{diag}} & \Sym_2(V) \\
   \Sym_2(\R^n) \arrow{r}{\Phi} & \Sym_2(W) \arrow{u}{\textrm{S}_2P}
  \end{tikzcd}\] commutes. Here $\textrm{diag}(a)$ denotes the diagonal matrix with diagonal $a\in\R^n$ and $P:W_{n-d,d}\to V_{n-d,d}$ is the orthogonal projection onto $V_{n-d,d}\subset W_{n-d,d}$.
 
Now if $\Phi(\diag(a))$ is positive semidefinite for some $a\in\R^n$, then its compression $\varphi(a)=(\textrm{S}_2P)(\Phi(\diag(a)))$ is positive semidefinite as well. For proving the converse of this statement, namely that $\Phi(\diag(a))$ is positive semidefinite whenever $\varphi(a)$ is positive semidefinite, we need some more careful analysis of the spaces $W_{n-d,d}$. This will result in a decomposition of $W$ as a direct sum of subspaces that are pairwise orthogonal with respect to $\Phi(\diag(a))$ for all $a\in\R^n$. In order to show that $\Phi(\diag(a))$ is positive semidefinite it then suffices to show that the restriction of $\Phi(\diag(a))$ to each of these subspaces is positive semidefinite. This will be done by identifying these restrictions of  $\Phi(\diag(a))$ with restrictions of $\varphi(a)$ to suitable subspaces of $V$.

\subsection{A decomposition of $W_{n-d,d}$}
In this section we always let $J_1,J_2\subset[n]$ be disjoint such that $|J_1|=|J_2|$ and $J_1\leq J_2$. We write $T^{d,n}_{J_1,J_2}\subset\Sym_2(\wedge^d\R^n)$ for the span of all $t_{I,J_1,J_2}$ with $I\in\binom{[n]}{d-|J_1|}$ disjoint from $J_1$ and $J_2$.

\begin{lem}\label{lem:orth}
 The subspaces $T^{d,n}_{J_1,J_2}\subset\Sym_2(\wedge^d\R^n)$ are pairwise orthogonal with respect to the bilinear form $b_{\cL_d(D)}$ for all diagonal matrices $D\in\Sym_2(\R^n)$.
\end{lem}

\begin{proof}
 This follows directly from \Cref{lem:orth1}.
\end{proof}

\begin{lem}\label{lem:tisclosed}
 If $A\in\Sym_2(\R^n)$ is the diagonal matrix with diagonal $a\in\R^n$, then
 $\Delta_A(t_{I,J_1,J_2})=\sum_{1\leq i\leq n,\, i\in I}a_it_{I\setminus\{i\},J_1,J_2}$.
 Moreover, the image of $T^{d,n}_{J_1,J_2}$ under $\Delta_A$ is contained in $T^{d-1,n}_{J_1,J_2}$.
\end{lem}

\begin{proof}
 We have seen in the proof of \Cref{lem:derisagree} that $$\psi_{\delta_i}(e_S)=\begin{cases}\pm e_{S\setminus\{i\}},\, i\in T\\ 0, j\not\in T \end{cases}.$$
 Thus we have that $(\psi_{\delta_i}\otimes\psi_{\delta_i})(e_S\otimes e_T)=e_{S\setminus\{i\}}\otimes e_{T\setminus\{i\}}$ if $i\in S\cap T$ and $0$ otherwise. Applying this to the definition of the $t_{I,J_1,J_2}$ and using the linearity of the map $X\mapsto\Delta_X$ shows that $\Delta_A(t_{I,J_1,J_2})=\sum_{1\leq i\leq n,\, i\in I}a_it_{I\setminus\{i\},J_1,J_2}$. The second claim is a direct consequence of the first claim.
\end{proof}

We denote by $W_{n-d,d}(J_1,J_2)$ the intersection of $W_{n-d,d}$ with $T^{d,n}_{J_1,J_2}$.

\begin{cor}
 $W_{n-d,d}$ is the direct sum of all $W_{n-d,d}(J_1,J_2)$.
\end{cor}

\begin{proof}
 Because $\Sym_2(\wedge^d\R^n)$ is the direct sum of the $T_{J_1,J_2}^{d,n}$ we only have to show that $W_{n-d,d}$ is spanned by the $W_{n-d,d}(J_1,J_2)$. To this end write $a\in W_{n-d,d}$  as $a=\sum_{J_1,J_2}a_{J_1,J_2}$ for some $a_{J_1,J_2}\in T_{J_1,J_2}^{d,n}$. By definition of $W_{n-d,d}$ we have $$0=\Delta_I(a)=\sum_{J_1,J_2}\Delta_I(a_{J_1,J_2}).$$ \Cref{lem:tisclosed} implies that $\Delta_I(a_{J_1,J_2})\in T_{J_1,J_2}^{d-1,n}$. But since $\Sym_2(\wedge^{d-1}\R^n)$ is the direct sum of the $T_{J_1,J_2}^{d-1,n}$, this shows $\Delta_I(a_{J_1,J_2})=0$ and thus $a_{J_1,J_2}\in W_{n-d,d}(J_1,J_2)$.
\end{proof}

Our next goal is to prove that the decomposition of $W_{n-d,d}$ into the direct sum of all $W_{n-d,d}(J_1,J_2)$ behaves well with respect to the bilinear forms defined in the previous section.

\begin{cor}\label{cor:orthbzglab}
 Let $A\in\Sym_2(\R^n)$ be the diagonal matrix with diagonal $a\in\R^n$.
 The subspaces $W_{n-d,d}(J_1,J_2)\subset\Sym_2(\wedge^d\R^n)$ are pairwise orthogonal with respect to the bilinear form $b_{\cL_d(A)}$.
\end{cor}

\begin{proof}
 This follows from \Cref{lem:orth}.
\end{proof}

\begin{cor}\label{cor:wisclosed}
 The image of $W_{n-d,d}(J_1,J_2)$ under the map $\Delta_A$ is contained in $W_{n-d+1,d-1}(J_1,J_2)$.
\end{cor}

\begin{proof}
 This follows from \Cref{lem:tisclosed} and $\Delta_A\circ\Delta_I=\Delta_I\circ\Delta_A$.
\end{proof}

\begin{cor}\label{cor:orthbzglc}
 Let $A\in\Sym_2(\R^n)$ be the diagonal matrix with diagonal $a\in V_{n-1,1}\subset\R^n$. Let $J^i_1,J^i_2\subset[n]$, $i=1,2$, be two different pairs of disjoint subsets such that $|J^i_1|=|J^i_2|=k$ and $J^i_1\leq J^i_2$ for $i=1,2$. Then the subspaces $W_{n-d,d}(J_1^1,J_2^1)$ and $W_{n-d-1,d+1}(J_1^2,J_2^2)$ are orthogonal with respect to the bilinear form  $C_{d,n}(A)$.
\end{cor}

\begin{proof}
 By definition of $C_{d,n}(A)$, this is a direct consequence of \Cref{cor:wisclosed}.
\end{proof}

Now we construct an embedding of $W_{n-d,d}(J_1,J_2)$ to $V_{n-d,d}$ which is compatible with the bilinear forms from the previous section. This allows us to deduce positive semidefiniteness on $W_{n-d,d}(J_1,J_2)$ from positive semidefiniteness on $V_{n-d,d}$.

\vspace{4pt}

Let $k=|J_1|=|J_2|$ and write $J_1=\{j_1,\ldots,j_k\}$ and $J_2=\{j'_1,\ldots,j'_k\}$ with $j_1<\cdots<j_k$ and $j'_1<\cdots<j'_k$. For $k\leq d\leq n$ consider the linear map defined by:

$$\rho_{d,n}: T_{J_1,J_2}^{d,n}\to\Ma_{d,n},\, t_{I,J_1,J_2}\mapsto (x_{j_1}-x_{j_1'})\cdots(x_{j_k}-x_{j_k'})  \cdot  \prod_{i\in I} x_i .$$

\begin{lem}
 The image of $W_{n-d,d}(J_1,J_2)$ under $\rho_{d,n}$ is contained in $V_{n-d,d}\subset W_{n-d,d}$.
\end{lem}

\begin{proof}
 Let $\alpha\in W_{n-d,d}(J_1,J_2)$. Then we can write $$\rho_{d,n}(\alpha)=(x_{j_1}-x_{j_1'})\cdots(x_{j_k}-x_{j_k'})\cdot f$$ for some multiaffine polynomial $f$ of degree $d-2k$ in the variables indexed by $I$. By \Cref{lem:tisclosed} we have that $$0=\rho_{d-1,n}(\Delta_I\alpha)=(x_{j_1}-x_{j_1'})\cdots(x_{j_k}-x_{j_k'})\cdot \De_ef.$$This implies that $\De_e f=0$ since $J_1$ and $J_2$ are disjoint.
 But since we have for all $1\leq l\leq k$ that $\De_e (x_{j_l}-x_{j_l'})=0$, the derivative of the entire product in direction of $e$ vanishes, meaning that is contained in $W_{n-d,d}$.
\end{proof}

Recall that we consider $\Ma_{d,n}$ as an $\fS_n$-invariant subspace of $\Sym_2(\wedge^d\R^n)$ as in \Cref{eqn:mainsym}.

\begin{lem}\label{lem:orth2}
 Let $D\in\Sym_2(\R^n)$ be a diagonal matrix. Let $I,I'\in\binom{[n]\setminus(J_1\cup J_2)}{d-k}$. If $I\neq I'$, then $b_{\cL_d(D)}(\rho_{d,n}(t_{I,J_1,J_2}),\rho_{d,n}(t_{I',J_1,J_2}))=0$. 
\end{lem}

\begin{proof}
 This follows directly from the fact that the set of monomials that appear in $\rho_{d,n}(t_{I,J_1,J_2})$ is disjoint from the set of monomials in $\rho_{d,n}(t_{I',J_1,J_2})$.
\end{proof}

\begin{lem}\label{lem:min1}
 Let $A\in\Sym_2(\R^n)$ be the diagonal matrix with diagonal $a\in\R^n$.
 For all $f\in T_{J_1,J_2}^{d,n}$ we have that $$2^{k}\cdot b_{\cL_d(A)}(f,f)= b_{\cL_d(A)}(\rho_{d,n}(f),\rho_{d,n}(f)).$$
\end{lem}

\begin{proof}
 It suffices to show the claim for $f=t_{I,J_1,J_2}$ with $I\in\binom{[n]}{d-|J_1|}$ as these elements are a basis of $T_{J_1,J_2}^{d,n}$ orthogonal with respect to $b_{\cL_d(A)}$ (\Cref{lem:orth1}) and their images under $\rho_{d,n}$ are pairwise orthogonal with respect to $b_{\cL_d(A)}$ as well (\Cref{lem:orth2}). Since $\rho_{d,n}$ is the identity if $J_1=\emptyset$, we can assume $k=|J_1|>0$. Then we have $$b_{\cL_d(A)}(t_{I,J_1,J_2},t_{I,J_1,J_2})= \tr(t_{I,J_1,J_2} \cdot \cL_d(A)  \cdot t_{I,J_1,J_2})=\sum_{i\in I} a_i+\frac{1}{2}\sum_{j\in J_1\cup J_2} a_j.$$ On the other hand, the polynomial $\rho_{d,n}(f)=(x_{j_1}-x_{j_1'})\cdots(x_{j_k}-x_{j_k'})  \cdot  \prod_{i\in I} x_i$ consists of $2^k$ monomials all of whose coefficients are $\pm 1$. Each $x_i$ for $i\in I$ appears in every such monomial and each $x_j$ for $j\in J_1\cup J_2$ appears in $2^{k-1}$ of those. Thus we have \[b_{\cL_d(A)}(\rho_{d,n}(f),\rho_{d,n}(f))=2^k\sum_{i\in I} a_i+2^{k-1}\sum_{j\in J_1\cup J_2} a_j.\qedhere\]
\end{proof}

\begin{lem}\label{lem:min2}
 Let $A\in W_{n-1,1}$ be the diagonal matrix with diagonal $a\in V_{n-1,1}$ and consider the bilinear form $C_{d,n}(A)\in W_{n-d,d}\otimes W_{n-d-1,d+1}$. 
 For all $f\in W_{n-d,d}(J_1,J_2)$ and $g\in W_{n-d-1,d+1}(J_1,J_2)$ we have that $$2^{k}\cdot C_{d,n}(A) (f,g)= C_{d,n}(A)(\rho_{d,n}(f),\rho_{d+1,n}(g)).$$
\end{lem}

\begin{proof}
 $C_{d,n}(A) (f,g)$ is defined to be the scalar product $\langle f,\Delta_A(g)\rangle$ of $f$ with $\Delta_A(g)$.  Since $\Delta_A$ is linear in $A$, it suffices to show $$2^k\cdot\langle f,\Delta_A(g)\rangle=\langle \rho_{d,n}(f),\Delta_A(\rho_{d+1,n}(g))\rangle$$ for $A$ the diagonal matrix whose diagonal $a=\delta_i$ is the $i$th unit vector, and all $f\in T^{d,n}_{J_1,J_2}$, $g\in T^{d+1,n}_{J_1,J_2}$. 
 In this case, using \Cref{lem:tisclosed}, we have that
 $$\langle t_{I,J_1,J_2},\Delta_A(t_{I',J_1,J_2})\rangle=\begin{cases}
                                            1 & \textrm{if } I=I'\setminus i, \\ 0 & \textrm{otherwise.}
                                           \end{cases}$$
 Therefore, we have to show that
 $$\langle \rho_{d,n}(t_{I,J_1,J_2}),\Delta_A(\rho_{d+1,n}(t_{I',J_1,J_2}))\rangle=\begin{cases}
                                            2^k & \textrm{if } I=I'\setminus i, \\ 0 & \textrm{otherwise.}
                                           \end{cases}$$
 In the case $I=I'\setminus i$ we have that $$\rho_{d,n}(t_{I'\setminus i,J_1,J_2})=\Delta_A(\rho_{d+1,n}(t_{I',J_1,J_2}))$$ is a multiaffine polynomial with exactly $2^k$ monomials all of whose coefficients are $\pm1$. This shows that $\langle \rho_{d,n}(t_{I'\setminus i,J_1,J_2}),\Delta_A(\rho_{d+1,n}(t_{I',J_1,J_2}))\rangle=2^k$. If $I\neq I'\setminus i$, then there is a $l\in I'\setminus(I\cup\{i\})$. Every monomial of $\Delta_A(\rho_{d+1,n}(t_{I',J_1,J_2}))$ but no monomial of $t_{I,J_1,J_2}$ is divisible by $x_l$. Thus their scalar product is zero.
\end{proof}

\section{Proof of the main theorem}\label{sec:mainres}
Let $V$ be a short representation of $\fS_n$ and $\varphi:\R^n\to\Sym_2(V)$ a $\fS_n$-linear map. We denote by $\varphi_i$ the restriction of $\varphi$ to the submodule of $\R^n$ isomorphic to $V_{n-i,i}$, $i=0,1$. We can write $V$ as a direct sum $\oplus_{j=1}^{m} V_{j}$ of irreducible $\fS_n$-submodules where each $V_j$ is isomorphic to $V_{n-\epsilon(j),\epsilon(j)}$ with $0\leq\epsilon(j)\leq\floor{\frac{n}{2}}$. After relabeling we can assume that $\epsilon(j_1)\leq\epsilon(j_2)$ if $j_1\leq j_2$. Then we have that $$\Sym_2(V)=\bigoplus_{j=1}^m\Sym_2(V_j)\oplus\bigoplus_{1\leq k<l\leq m} (V_k\otimes V_l).$$Each of these summands contains at most one copy of $V_n$ and $V_{n-1,1}$ by \Cref{lem:stdinsym} and \Cref{lem:stdintens}. Therefore, the map $\varphi_0:V_n\to\Sym_2(V)$ is the direct sum of the maps $a_{jj}\cdot\alpha_{\epsilon(j),n}:V_n\to\Sym_2(V_j)$ and $a_{kl}\cdot\alpha_{\epsilon(k),n}:V_n\to V_k\otimes V_l$ (if $k<l$ and $\epsilon(k)=\epsilon(l)$) for suitable real numbers $a_{jj}$ and $a_{kl}$. Analogously, $\varphi_1$ is the direct sum of the maps $b_{jj}\cdot\beta_{\epsilon(j),n}:V_{n-1,1}\to\Sym_2(V_j)$, $b_{kl}\cdot\beta_{\epsilon(k),n}:V_{n-1,1}\to V_k\otimes V_l$ (if $k<l$ and $\epsilon(k)=\epsilon(l)$) and $c_{kl}\cdot\gamma_{\epsilon(k),n}:V_{n-1,1}\to V_k\otimes V_l$ (if $k<l$ and $\epsilon(k)+1=\epsilon(l)$) for suitable real numbers $b_{jj}$, $b_{kl}$ and $c_{kl}$.

From this we define the $\Or(n)$-module $W$ as $\oplus_{j=1}^{m} W_{j}$ where $W_j$ is an $\Or(n)$-module isomorphic to $W_{n-\epsilon(j),\epsilon(j)}$. We define $\Phi: \Sym_2(\R^n)\to\Sym_2(W)$ as the direct sum of the maps $\Phi_i: W_{n-i,i}\to\Sym_2(W)$, $i=0,1$, which are defined as follows. The map $\Phi_0:W_n\to\Sym_2(W)$ is defined to be the direct sum of the maps $a_{jj}\cdot A_{\epsilon(j),n}:W_n\to\Sym_2(W_j)$ and $a_{kl}\cdot A_{\epsilon(k),n}:W_n\to W_k\otimes W_l$ (if $k<l$ and $\epsilon(k)=\epsilon(l)$). Analogously, $\Phi_1$ is the direct sum of the maps $b_{jj}\cdot B_{\epsilon(j),n}:W_{n-1,1}\to\Sym_2(W_j)$, $b_{kl}\cdot B_{\epsilon(k),n}:W_{n-1,1}\to W_k\otimes W_l$ (if $k<l$ and $\epsilon(k)=\epsilon(l)$) and $c_{kl}\cdot C_{\epsilon(k),n}:W_{n-1,1}\to W_k\otimes V_l$ (if $k<l$ and $\epsilon(k)+1=\epsilon(l)$). The map $\Phi: \Sym_2(\R^n)\to\Sym_2(W)$ is $\Or(n)$-linear by construction. 

\begin{ex}
 If $V=\Ma_{d,n}$, then $W\cong\Sym_2(\wedge^d\R^n)$.
\demo \end{ex}

Further the inclusion $V_{n-d,d}\subset W_{n-d,d}$ defined in \Cref{sec:on} induces an inclusion $V\subset W$. We have:

\begin{prop}\label{prop:resphi}
 Let $\Lambda\in\Sym_2(\R^n)$ be the diagonal matrix with diagonal $\lambda\in\R^n$. Then the restriction of the bilinear form $\Phi(\Lambda)$ to $V$ is $\varphi(\lambda)$.
\end{prop}

\begin{proof}
 This follows by construction from \Cref{lem:resab} and \Cref{lem:resc}.
\end{proof}

\begin{cor}\label{cor:easydir}
 Let $X\in\Sym_2(\R^n)$ and $\lambda\in\R^n$ be the vector of eigenvalues of $X$. If $\Phi(X)$ is positive semidefinite, then $\varphi(\lambda)$ is positive semidefinite.
\end{cor}

\begin{proof}
 Let $\Lambda\in\Sym_2(\R^n)$ be the diagonal matrix with diagonal $\lambda$ and let $S\in\Or(n)$ be an orthogonal matrix such that $S^tX S=\Lambda$. Since $\Phi$ is $\Or(n)$-linear, $\Phi(X)$ being positive semidefinite implies that $\Phi(\Lambda)$ is positive semidefinite. But then its restriction to $V$, which is $\varphi(\lambda)$ by \Cref{prop:resphi}, is also positive semidefinite.
\end{proof}

In order to show the other direction, we decompose $W$ into a direct sum of linear subspaces that are pairwise orthogonal with respect to the bilinear form $\Phi(\Lambda)$ for every diagonal matrix $\Lambda\in\Sym_2(\R^n)$. As an $\Or(n)$-module, $W$ equals $\oplus_{j=1}^{m} W_{j}$ where $W_j$ is an $\Or(n)$-module isomorphic to $W_{n-\epsilon(j),\epsilon(j)}$. Fix disjoint $J_1,J_2\subset[n]$ such that $|J_1|=|J_2|$ and $J_1\leq J_2$. We have defined the subspace $$W_{n-\epsilon(j),\epsilon(j)}(J_1,J_2)\subset W_{n-\epsilon(j),\epsilon(j)}\cong W_j\subset W.$$ The direct sum of these subspaces of $W$ for all $j$ is denoted by $W(J_1,J_2)$.

\begin{lem}\label{lem:wareorth}
 The subspaces $W(J_1,J_2)$ are pairwise orthogonal with respect to the bilinear form $\Phi(\Lambda)$ on $W$ for every diagonal matrix $\Lambda\in\Sym_2(\R^n)$.
\end{lem}

\begin{proof}
 The bilinear form $\Phi(\Lambda)$ is the sum of bilinear forms with respect to which the subspaces $W(J_1,J_2)$ are pairwise orthogonal by \Cref{cor:orthbzglab} and \Cref{cor:orthbzglc}.
\end{proof}

\begin{lem}\label{lem:warepsd}
 Let $\Lambda\in\Sym_2(\R^n)$ be the diagonal matrix with diagonal $\lambda\in\R^n$.
 If $\varphi(\lambda)$ is positive semidefinite, then the restriction of $\Phi(\Lambda)$ to $W(J_1,J_2)$ is positive semidefinite as well.
\end{lem}

\begin{proof}
 The combination of \Cref{lem:min1} and \Cref{lem:min2} shows that the restriction of $\Phi(\Lambda)$ to $W(J_1,J_2)$ is a positive scalar multiple of the restriction of $\Phi(\Lambda)$ to a certain subspace of $V$. Since the restriction of $\Phi(\Lambda)$ to $V$ is $\varphi(\lambda)$ by \Cref{prop:resphi}, this implies the claim.
\end{proof}

\begin{thm}\label{thm:psdiff}
 Let $X\in\Sym_2(\R^n)$ and $\lambda\in\R^n$ be the vector of eigenvalues of $X$. Then $\Phi(X)$ is positive semidefinite if and only if $\varphi(\lambda)$ is positive semidefinite.
\end{thm}

\begin{proof}
 One direction was shown in \Cref{cor:easydir}. For the other direction let $\Lambda\in\Sym_2(\R^n)$ be the diagonal matrix with diagonal $\lambda$ and let $S\in\Or(n)$ be an orthogonal matrix such that $S^tX S=\Lambda$. Since $\Phi$ is $\Or(n)$-linear, $\Phi(X)$ being positive semidefinite is equivalent to $\Phi(\Lambda)$ being positive semidefinite. The latter is the case if $\varphi(\lambda)$ is positive semidefinite by \Cref{lem:wareorth} and \Cref{lem:warepsd}.
\end{proof}

\begin{ex}\label{ex:firstderi}
 Let $V=V_{n-1,1}$, i.e., $m=1$ and $\epsilon(1)=1$, and let $\varphi:\R^n\to\Sym_2(V)$ be the map obtained by composing the diagonal map $\R^n\to\Sym_2(\R^n)$ with the restriction to $V_{n-1,1}$, i.e., we have $a_{11}=b_{11}=1$. It was shown in \cite{pol1} that $\varphi(a)$ is positive semidefinite if and only if $a$ is in the hyperbolicity cone of $\sigma_{n-1}([n])$. The associated $\Or(n)$-module is then $W=W_{n-1,1}$, the space of traceless symmetric matrices. The map $\Phi: \Sym_2(\R^2)\to\Sym_2(W)$ is obtained by composing the diagonal map $\Sym_2(\R^n)\to\Sym_2(\Sym_2(\R^n))$ with the restriction to $W_{n-1,1}$. This is Saunderson's spectrahedral representation of the first derivative relaxation of the positive semidefinite cone \cite{firstderi}.
\demo \end{ex}

\begin{ex}
 We want to carry out one completely explicit example. In order to avoid very large matrices, we consider $V=\R^3=V_3\oplus V_{2,1}$. Similarly to \Cref{ex:ma24complete}, for every $\fS_3$-linear $\varphi:\R^3\to\Sym_2(\R^3)$ the matrix $\varphi(a)$ equals to
 $$\bordermatrix{
	  &V_3 &  V_{2,1}  \cr
V_3 & \lambda_1 M_1(a)  & \lambda_2 M_2(a)   \cr
V_{2,1} & \lambda_2 M_2(a)^t &  \lambda_3M_3(a)+\lambda_4M_4(a)   \cr
}
 $$for some $\lambda_1,\ldots,\lambda_4\in\R$ where we define
 $$M_1(a)=\left(a_1+a_2+a_3\right), \, M_2(a)=(a_1-a_2,a_1-a_3),$$
 $$M_3(a)=(a_1+a_2+a_3)\cdot\begin{pmatrix}
                             2&1\\ 1&2
                            \end{pmatrix},\, M_4(a)=\begin{pmatrix}
                             a_1+a_2&a_1\\ a_1&a_1+a_3
                            \end{pmatrix}.$$
 For example when $\lambda_1=\cdots=\lambda_4=1$, we obtain $$\varphi(a)=\begin{pmatrix}
                                                                          a_1+a_2+a_3&  a_1-a_2&  a_1-a_3 \\a_1-a_2&  3 a_1+3 a_2+2 a_3&  2 a_1+a_2+a_3 \\a_1-a_3&  2 a_1+a_2+a_3&  3 a_1+2 a_2+3 a_3
                                                                         \end{pmatrix}.$$
 The matrix $\varphi(a)$ is positive semidefinite if and only if $a$ is in the hyperbolicity cone of the irreducible ternary cubic polynomial $$h=\sigma_{1,3}^3+2 \sigma_{1,3} \sigma_{2,3}+3 \sigma_{3,3}.$$ In order to compute the corresponding $\Or(3)$-linear map we let $W=W_3\oplus W_{2,1}$ where $W_3$ is the trivial representation and $W_{2,1}$ is the space of symmetric traceless $3\times3$ matrices. As a basis of $W_{2,1}$ we choose $$E_{11}-E_{22}, E_{11}-E_{33}, E_{12}+E_{21}, E_{13}+E_{31}, E_{23}+E_{32}$$ where $E_{ij}$ is the matrix which whose $(i,j)$th entry is $1$ and all others are zero. Letting $\Phi:\Sym_2(\R^3)\to\Sym_2(\Sym_2(\R^3))$ the $\Or(3)$-linear map associated to $\varphi$, we have that $\Phi(A)$ equals to the matrix
 $$\bordermatrix{
	  &W_3 &  W_{2,1}  \cr
W_3 & \lambda_1 N_1(A)  & \lambda_2 N_2(A)   \cr
W_{2,1} & \lambda_2 N_2(A)^t &  \lambda_3N_3(A)+\lambda_4N_4(A)   \cr
}
 $$for the same $\lambda_1,\ldots,\lambda_4\in\R$, where we define
 $$N_1(A)=\left(\tr(A)\right), \, N_2(A)=(a_{11}-a_{22},a_{11}-a_{33},2a_{12},2a_{13},2a_{23}),$$
 $$N_3(A)=\tr(A)\cdot\begin{pmatrix}
                             2&1&0&0&0\\ 1&2&0&0&0\\0&0&2&0&0\\0&0&0&2&0\\ 0&0&0&0&2
                            \end{pmatrix},$$ $$\textrm{ and }\,\,\,\, N_4(A)=\begin{pmatrix}
                             a_{11}+a_{22}& a_{11}& 0& a_{13}& -a_{23} \\ a_{11}& a_{11}+a_{33}& a_{12}& 0& -a_{23} \\ 0& a_{12}& a_{11}+a_{22}& a_{23}& a_{13} \\ a_{13}& 0& a_{23}& a_{11}+a_{33}& a_{12} \\ -a_{23}& -a_{23}& a_{13}& a_{12}& a_{22}+a_{33}
                            \end{pmatrix}.$$
In particular when $\lambda_1=\cdots=\lambda_4=1$, we obtain for $\Phi(A)$ the matrix
$$\left(\begin{smallmatrix}
   \tr(A)& a_{11}-a_{22}& a_{11}-a_{33}& 2 a_{12}& 2 a_{13}& 2 a_{23}\\ a_{11}-a_{22}& 3\tr(A) - a_{33}&  a_{11}+\tr(A)& 0& a_{13}& -a_{23}\\ a_{11}-a_{33}& a_{11}+\tr(A)& 3\tr(A) - a_{22}& a_{12}& 0& -a_{23}\\ 2 a_{12}& 0& a_{12}& 3\tr(A) -a_{33}& a_{23}& a_{13}\\ 2 a_{13}& a_{13}& 0& a_{23}& 3\tr(A) -a_{22}& a_{12}\\ 2 a_{23}& -a_{23}& -a_{23}& a_{13}& a_{12}& 3\tr(A) -a_{11}
  \end{smallmatrix}\right).$$The matrix $\Phi(A)$ is positive semidefinite if and only if $A$ is in the hyperbolicity cone of the irreducible cubic polynomial $$H=P_{1}^3+2 P_{1} P_{2}+3 P_{3}$$where $P_i(A)=\sigma_{i,3}(\lambda(A))$. Its determinant however is the reducible sextic $$3\cdot(P_{1}^3+2 P_{1} P_{2}+3 P_{3})\cdot(18P_1^3+3P_1P_2-P_3) .\demo $$
\end{ex}

\begin{figure}[h]
  \begin{center}
    \includegraphics[width=4cm]{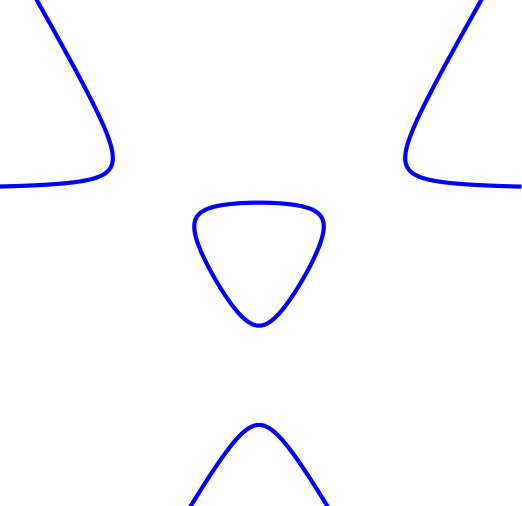} \quad
    \includegraphics[width=4cm]{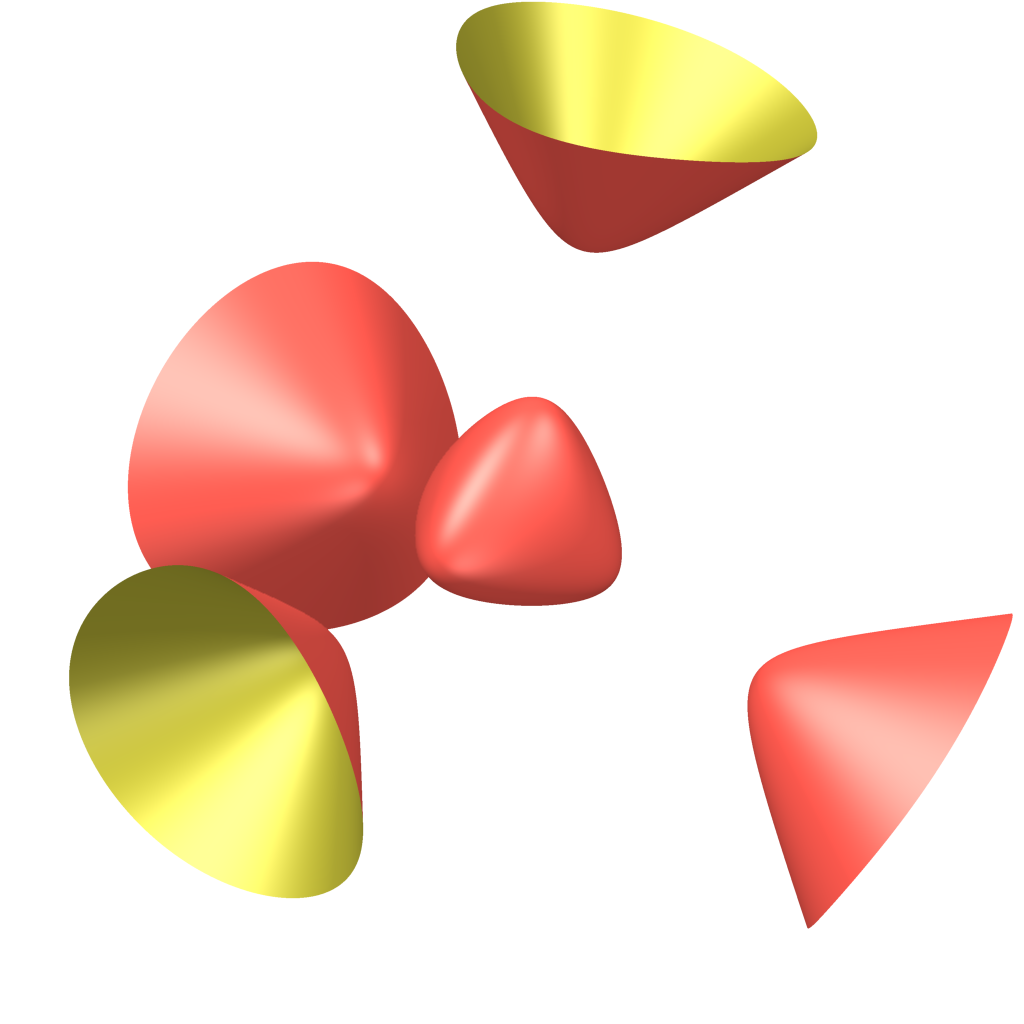}    
  \end{center}
  \caption{The zero set of the hyperbolic polynomial $h$ restricted to the affine hyperplane $x_1+x_2+x_3=1$ (left) and the zero set of the hyperbolic polynomial $H$ restricted to the space of symmetric $3\times3$ matrices with diagonal $(1,1,1)$ (right).}
  \label{fig:cubic2}
\end{figure}

\begin{rem}\label{rem:minors}
 Denote by $\R[\Sym_2(\R^n)]_d$ the space of homogeneous polynomial functions of degree $d$ on $\Sym_2(\R^n)$ and let $\Min_{d,n}$ be the subspace spanned by the $d\times d$ minors. Clearly, $\Min_{d,n}$ is an $\Or(n)$-invariant subspace. We consider a symmetric matrix $X\in\Sym_2(\R^n)$ as a self-adjoint endomorphism $\R^n\to\R^n$ and denote by $\wedge^dX$ the induced self-adjoint endomorphism of $\wedge^d\R^n$. To every self-adjoint endomorphism $A\in\Sym_2(\wedge^d\R^n)$ we associate the map $$f_A: \Sym_2(\R^n)\to\R,\, X\mapsto\tr(A\cdot\wedge^d X).$$ It is direct to see that $f_A\in\Min_{d,n}$ and that the map $$\Sym_2(\wedge^d\R^n)\to\Min_{d,n},\,A\mapsto f_A$$ is surjective and $\Or(n)$-invariant. The kernel of this map is orthogonal to $\Ma_{d,n}\subset\Sym_2(\wedge^d\R^n)$. If one cares about the size of spectrahedral representations, then one can further compress it to the subspace $\tilde{W}$ of $W$ that is obtained by replacing each summand $W_j\cong W_{n-\epsilon(j),\epsilon(j)}$ by its intersection $\tilde{W}_{n-\epsilon(j),\epsilon(j)}$ with the orthogonal complement of the kernel of the above map for $d=\epsilon(j)$. Since $\tilde{W}$ contains $V$, the resulting linear matrix inequality still describes the same spectrahedral cone.
\end{rem}

The content of the following remark is not needed for the rest of this article but it might give a more complete picture from the representation theoretic point of view. We therefore omit the proofs.

\begin{rem}
 We can describe the $\Or(n)$-module $\tilde{W}_{n-d,d}$ for $0\leq 2d\leq n$ as the kernel of the map $\De_I: \Min_{d,n}\to \Min_{d-1,n}$ where $\De_I$ denotes the derivative in direction of the identity matrix. The orthogonal projection from $\tilde{W}_{n-d,d}$ to the subspace $V_{n-d,d}$ is given by restricting a polynomial $p\in\tilde{W}_{n-d,d}$ to the space of diagonal matrices. Using the notation of \cite[\S 10.2.1]{goodwallach} the $\Or(n)$-module $\tilde{W}_{n-d,d}$ is isomorphic to the irreducible representation $E^{(d,d)'}$ where $(d,d)'$ is the partition $(2,\ldots,2)$ of $2d$. This suggests that the true analogon to the irreducible $\fS_n$-module $V_{n-d,d}$ should be the $\Or(n)$-module $\tilde{W}_{n-d,d}\cong E^{(d,d)'}$ rather than ${W}_{n-d,d}$. However, we think that the proof of \Cref{thm:psdiff} is easier to carry out by using ${W}_{n-d,d}$.
\end{rem}

\section{Derivative cones}\label{sec:deris}
Let us now recall the spectrahedral representation of the hyperbolicity cone of the elementary symmetric polynomial $\sigma_{d,n}$ that was constructed in \cite{BranEle}. From this we will construct a spectrahedral representation that satifies the conditions of the main theorem (\Cref{thm:main}). Note that an alternative spectrahedral representation of these hyperbolicity cones was constructed in \cite[Thm.~2.19]{matching}. We will compare the sizes of different such representations in \Cref{rem:sizes}.

Let $B(x)$ be a symmetric matrix whose rows and columns are indexed by words $w_1\ldots w_l$ with letters $w_i\in\{1,\ldots,n\}$ of length $0\leq l\leq d-1$ such that $w_i\neq w_j$ for $i\neq j$. We let the diagonal entry of $B(x)$ corresponding to the word $w_1\ldots w_l$ be $$(d-1-l)!\cdot ((d-l)\cdot x_{w_l}+\sum_{j\in[n]\setminus\{w_1,\ldots,w_l\}}x_j)$$and the entries $(w_1\ldots w_{l-1},w_1\ldots w_l)$ resp. $(w_1\ldots w_{l},w_1\ldots w_{l-1})$ for $1\leq l\leq d-1$ by $-(d-l)!\cdot x_{w_l}$. We set all other entries of $B(x)$ to zero. Then we have:

\begin{thm}[\cite{BranEle}]\label{thm:branden}
 A point $a\in\R^n$ is in the hyperbolicity cone of $\sigma_{d,n}$ if and only if $B(a)$ is positive semidefinite. More precisely, the determinant $\det B(x)$ is the product of $\sigma_{d,n}$ with a (nonzero) hyperbolic polynomial whose hyperbolicity cone contains the hyperbolicity cone of $\sigma_{d,n}$.
\end{thm}

\begin{ex}
 For $d=2$ and $n=4$ the matrix $B(x)$ is given by $$\bordermatrix{
	  &\emptyset &  1 & 2 &3 \cr
\emptyset & x_1+x_2+x_3  & -x_1 & -x_2 & -x_3  \cr
1 & -x_1 & x_1+x_2+x_3 & 0 & 0  \cr
2 &  -x_2 &  0 & x_1+x_2+x_3 & 0 \cr
3 &  -x_3 &  0 & 0 & x_1+x_2+x_3  \cr
}.$$
\demo \end{ex}

Since the determinant of $B(x)$ is divisible by $\sigma_{d,n}$, it follows that $B(x)$ has a nontrivial kernel at every point where $\sigma_{d,n}$ vanishes. In order to explicitly describe such a kernel vector, recall that for every subset $S\subset [n]$ we denote by $\sigma_i(S)$ the elementary symmetric polynomial of degree $i$ in the variables indexed by elements from $S$.

\begin{lem}\label{lem:brandenkern}
 Let $m(x)$ be a vector whose entries are indexed by words $w_1\ldots w_l$ with letters $w_i\in\{1,\ldots,n\}$ of length $0\leq l\leq d-1$ such that $w_i\neq w_j$ for $i\neq j$. If we let the entry of $m(x)$ corresponding to the word $w_1\ldots w_l$ equal to $$\prod_{i=1}^l x_{w_i} \cdot \sigma_{d-1-l}([n]\setminus\{w_1,\ldots,w_l\}),$$ then $B(x)\cdot m(x)=\delta_{\emptyset} \cdot d!\cdot \sigma_d([n])$ where $\delta_{\emptyset}$ is the unit vector corresponding to the empty word.
\end{lem}

\begin{proof}
 We compute the entry of $B(x)\cdot m(x)$ indexed by the word $w_1\ldots w_l$ for $0\leq l\leq d-1$.
The corresponding row of $B(x)$ has the following non-zero entries at the columns indexed by $w_1\ldots w_{l-1}$, $w_1\ldots w_{l}$
and $w_1\ldots w_{l+1}$ for $w_{l+1}\in[n]\setminus \{w_1,\ldots,w_l\}$. If we denote $S=[n]\setminus\{w_1,\ldots,w_{l-1}\}$, these entries
contribute the following summands:
\begin{enumerate}
	\item $w_1\ldots w_{l-1}$ (note that for $l=0$ this case does not occur): \[-(d-l)!\cdot x_{w_l}\cdot \prod_{i=1}^{l-1}x_{w_i}\cdot \sigma_{d-l}(S).\]
	\item $w_1\ldots w_{l}$: \[(d-1-l)!\cdot((d-l-1)\cdot x_{w_l}+\sigma_1(S))\cdot\prod_{i=1}^lx_{w_l}\cdot
					\sigma_{d-1-l}( S\setminus\{w_l\} ).\]
	\item $w_1\ldots w_{l+1}$ (note that for $l=d-1$ this case does not occur): \[(d-1-l)!\cdot x_{w_{l+1}}\cdot\prod_{i=1}^{l+1}x_{w_i} \cdot \sigma_{d-2-l}(S\setminus\{w_l,w_{l+1}\}).\]
\end{enumerate}
Summing these up, we arrive at the following expression for the entry of $B(x)\cdot m(x)$ indexed by the word $w_1\ldots w_l$:
\[( \underbrace{-(d-l)\cdot\sigma_{d-l}(S)}_{(1)}+\underbrace{((d-1-l)\cdot x_{w_l}+\sigma_1(S))
  \cdot\sigma_{d-1-l}( S\setminus\{w_l\} )}_{(2)}\]\[-\underbrace{\sum_{j\in S\setminus\{w_l\}} x_j^2\cdot \sigma_{d-2-l}([n])
  \sigma_{d-2-l}(S\setminus\{w_l,j\})}_{(3)}) \cdot(d-1-l)!\cdot\prod_{i=1}^l x_{w_i}.\] Two types of monomials can appear inside the paranthesis of the above expression:
\begin{enumerate}[a)]
	\item 	The coefficient of a multiaffine monomial in the variables indexed by $S$ is $-(d-l)$ in $(1)$ and in $(2)$
		it is $(d-l)$.
	\item 	The coefficient of $x_j^2$ times a multiaffine monomial in variables indexed by $S\setminus\{w_l,j\}$ is $1$ in $(2)$ and
		$-1$ in $(3)$ in the case $j<d-1$. If $j=d-1$ such monomials do not appear at all.
\end{enumerate}
We can conclude that for $l=0$ we obtain $d!\cdot\sigma_d([n])$ and for $l>0$ we get zero.
\end{proof}

\begin{ex}
 For $d=2$ and $n=3$ the vector $m(x)$ is given by\vspace{-15pt} $$ \bordermatrix{
	  & \cr
\emptyset & x_1+x_2+x_3    \cr
1 & x_1   \cr
2 &  x_2  \cr
3 &  x_3   \cr
}.\demo$$
 \end{ex}

\begin{thm}\label{thm:kleinebranden}
 There is an $\fS_n$-linear map $\psi:\R^n\to\Sym_2(\Ma_{d-1,n})$ such that $\psi(a)$ is positive semidefinite if and only if $a$ is in the hyperbolicity cone of $\sigma_{d,n}$.
\end{thm}

\begin{proof}
 Let $B(x)$ be the spectrahedral representation of the hyperbolicity cone of $\sigma_{d,n}=\sigma_d([n])$ from \cite{BranEle} and $m(x)$ the vector from \Cref{lem:brandenkern}. We have $$B(x)\cdot m(x)=\delta_{\emptyset} \cdot d!\cdot \sigma_d([n])$$ according to \Cref{lem:brandenkern}. The entries of $m(x)$ span $\Ma_{d-1,n}$ as $\R$-vector space. Thus if $\tilde{m}(x)$ is a vector whose entries comprise a basis of $\Ma_{d-1,n}$, there is a unique rectangular real matrix $Q$ of full rank such that $Q\cdot\tilde{m}(x)=m(x)$. Letting $\tilde{B}(x)=\frac{1}{d!}Q^tB(x)Q$ and $v=Q^t\delta_\emptyset$ we obtain \begin{equation}\label{eq:compsym}
 \tilde{B}(x)\cdot \tilde{m}(x)=\frac{1}{d!}Q^tB(x)Q\tilde{m}(x)=\frac{1}{d!}Q^tB(x){m}(x)=Q^t\delta_{\emptyset} \cdot \sigma_d=v\cdot \sigma_d([n]).
 \end{equation}
 Furthermore, it is not hard to see that the map $\psi:a\mapsto\tilde{B}(a)$ is a homomorphism of $\fS_n$-modules if we consider $\tilde{B}(a)$ as an element of $\Sym_2(\Ma_{d-1,n})$ via the chosen basis of $\Ma_{d-1,n}$. But since $\tilde{B}(a)$ is a compression of $B(a)$, we have that $\tilde{B}(a)$ is positive semidefinite whenever $B(a)$ is positive semidefinite. On the other hand \Cref{eq:compsym} implies that $\tilde{B}(a)$ is singular whenever $\sigma_d([n])$ vanishes at $a$.
This implies that $\tilde{B}(a)$ cannot be positive semidefinite if $a$ is not in the hyperbolicity cone of $\sigma_d([n])$. 
\end{proof}

\begin{rem}\label{rem:sizes}
 The size of the spectrahedral representation from \cite{BranEle} is given by the number of words $w_1\ldots w_l$ with letters $w_i\in\{1,\ldots,n\}$ of length $0\leq l\leq d-1$ such that $w_i\neq w_j$ for $i\neq j$. The size of the spectrahedral representation from \cite{matching} is given by the number of simple paths of length (at least $1$ and) at most $d$ starting from a fixed vertex of the complete graph on $n$ vertices. This is the same as the number of words $w_1\ldots w_l$ with letters $w_i\in\{1,\ldots,n-1\}$ of length $0\leq l\leq d-1$ such that $w_i\neq w_j$ for $i\neq j$. Clearly, the latter is smaller than the former. Finally, the representation that we obtain in \Cref{thm:kleinebranden} is of size $\binom{n}{d-1}$, the number of multiaffine monomials of degree $d-1$ in $n$ variables. This is the same as the number of words $w_1\ldots w_l$ with letters $w_i\in\{1,\ldots,n-1\}$ of length $d-2\leq l\leq d-1$ such that $w_i< w_j$ for $i<j$. Thus it is considerably smaller than the other two representations.
\end{rem}

\begin{rem}\label{rem:explicitematrix}
 For future reference we want to make the entities that appear in \Cref{eq:compsym} explicit. To that end, we choose $\tilde{m}(x)$ to be the vector whose entries comprise the monomial basis of $\Ma_{d-1,n}$. Thus we can view $\tilde{m}(x)$ as a column vector whose rows are indexed by $\binom{[n]}{d-1}$ and the entry corresponding to $I\in\binom{[n]}{d-1}$ is the monomial $\prod_{i\in I}x_i$. The columns of the matrix $Q$ are indexed by $\binom{[n]}{d-1}$ and its rows are indexed by words $w_1\ldots w_l$ with letters $w_i\in\{1,\ldots,n\}$ of length $0\leq l\leq d-1$ such that $w_i\neq w_j$ for $i\neq j$. The entry of $Q$ indexed by $(w_1\cdots w_l,I)$ is $1$ if $w_1,\ldots,w_l\in I$ and zero otherwise. This implies that $v=Q^t\delta_\emptyset$ is the all-ones vector $e$. Finally, the rows and columns of the matrix $\tilde{B}(x)$ are both indexed by $\binom{[n]}{d-1}$. The entry of $\tilde{B}(x)$ indexed by $(I,J)$ equals $$\frac{1}{d-|I\cap J|}\sum_{k\in[n]\setminus(I\cup J)}x_k.$$
\end{rem}

\begin{ex}\label{ex:o23}
 For $d=2$ and $n=3$ the matrix $\tilde{B}(x)$ is given by $$\bordermatrix{
	  &  \{1\} & \{2\} &\{3\} \cr
\{1\} & x_2+x_3 & \frac{1}{2}x_3 & \frac{1}{2}x_2  \cr
\{2\} &  \frac{1}{2}x_3 &  x_1+x_3 & \frac{1}{2}x_1 \cr
\{3\} &  \frac{1}{2}x_2 &  \frac{1}{2}x_1 & x_1+x_2   \cr
}.$$It is a spectrahedral representation of the hyperbolicity cone of $\sigma_{2,3}$. Note that a smaller spectrahedral representation is the one from \cite{pol1} given by a $\fS_3$-linear map $\R^3\to\Sym_2(V_{2,1})$, see \Cref{ex:firstderi}.
\demo \end{ex}

\begin{ex}\label{ex:smallernonsymmetric}
 For $d=2$ and $n=4$ the matrix $\tilde{B}(x)$ is given by
 $$\bordermatrix{
	  &  \{1\} & \{2\} &\{3\}&\{4\} \cr
\{1\} & x_2+x_3+x_4 & \frac{1}{2}(x_3+x_4) & \frac{1}{2}(x_2+x_4) &  \frac{1}{2}(x_2+x_3) \cr
\{2\} &  \frac{1}{2}(x_3+x_4) &  x_1+x_3+x_4 & \frac{1}{2}(x_1+x_4) & \frac{1}{2}(x_1+x_3) \cr
\{3\} &  \frac{1}{2}(x_2+x_4) &  \frac{1}{2}(x_1+x_4) & x_1+x_2+x_4 & \frac{1}{2}(x_1+x_2)   \cr
\{4\} &  \frac{1}{2}(x_2+x_3) &  \frac{1}{2}(x_1+x_3) & \frac{1}{2}(x_1+x_2)  & x_1+x_2+x_3  \cr
}.$$It is a spectrahedral representation of the hyperbolicity cone of $\sigma_{2,4}$. There is a smaller spectrahedral representation:
$$A(x)=
\begin{pmatrix}
 x_1+x_2+x_4 & \frac{1}{2}x_2+x_4 & \frac{1}{2}x_1+x_4\\ \frac{1}{2}x_2+x_4 & x_2+x_3+x_4 & \frac{1}{2}x_3+x_4\\ \frac{1}{2}x_1+x_4 & \frac{1}{2}x_3+x_4 & x_1+x_3+x_4
\end{pmatrix}
.$$By \cite[\S3]{twores} there is no representation smaller than $A(x)$. Its determinant $$\det A(x)=\frac{3}{4}\cdot(x_1+x_2+x_3)\cdot\sigma_{2,4}$$is not invariant under the action of $\fS_4$. Therefore, this representation is not symmetric in the sense that it is not given by an $\fS_4$-linear map $\R^4\to\Sym_2(V)$ for some $\fS_4$-module $V$. In fact, we claim that if $V$ is a $3$-dimensional $\fS_4$-module, there is no $\fS_4$-linear map $\psi:\R^4\to\Sym_2(V)$ such that $\psi(a)$ is positive semidefinite if and only if $a$ is in the hyperbolicity cone of $\sigma_{2,4}$. Indeed, such $V$ must have the property that $\Sym_2(V)$ has $V_{3,1}$ as one of its irreducible components. The only $3$-dimensional $\fS_4$-modules with this property are $V_{3,1}$ and $V_{2,1,1}$. Since we have that $\Sym_2(V_{3,1})$ and $\Sym_2(V_{2,1,1})$ are isomorphic, we assume without loss of generality that $V=V_{3,1}$. Now let $\psi:\R^4\to\Sym_2(V_{3,1})$ be a map as above. Then the determinant of $\psi(x)$ is necessarily divisible by $\sigma_{2,4}$ and invariant under $\fS_4$. Thus it must also be divisible by $\sigma_{1,4}$. This shows that $\psi(x)$ is singular for all $x$ from the zero set $V_{3,1}\subset\R^4$ of $\sigma_{1,4}$. But one can check that every nonzero matrix in the $V_{3,1}$ component of $\Sym_2(V_{3,1})$ is nonsingular. This yields the desired contradiction.
\demo \end{ex}

\begin{cor}\label{cor:elemspec}
 The set of all symmetric matrices $X\in\Sym_2(\R^n)$ whose spectrum $\lambda(X)$ is in the hyperbolicity cone of $\sigma_{d+1,n}$ is a spectrahedral cone.
\end{cor}

\begin{proof}
 Since $\Ma_{d,n}$ is a short representation by \Cref{cor:mashort}, this follows from \Cref{thm:main} and \Cref{thm:kleinebranden}.
\end{proof}

\begin{rem}
 Using \Cref{rem:minors} we obtain a spectrahedral representation for this set whose size is the dimension of $\Min_{d,n}$. In order to determine this dimension we note that considering $\Min_{d,n}$ as a $\GL_n$-module, it is irreducible with highest weight $(d,d)'$, see for example \cite[Thm.~3.19]{weymanetal}. Thus by \cite[Thm.~6.3(1)]{fultonharris} we have $$\dim(\Min_{d,n})=\prod_{1\leq i\leq d}\left(\prod_{d+1\leq j \leq n} \frac{2+j-i}{j-i}\right)=\prod_{i=1}^d\frac{(n+1-i)(n+2-i)}{(d+1-i)(d+2-i)}.$$In particular, for fixed $d$, the dimension of $\Min_{d,n}$ grows only polynomially in $n$. We further note that $\dim(\Min_{d,n})=\dim(\Min_{n-d,n})$.
\end{rem}

\begin{cor}\label{cor:charapoly}
 Let $X$ be the \emph{generic} symmetric $n\times n$ matrix, i.e., its entries are given by the variables $x_{ij}$ for $1\leq i\leq j\leq n$. Then the hyperbolicity cone of every derivative $\De_I^d \det(X)$ is spectrahedral.
\end{cor}

\begin{proof}
 Let us write $$\det(tI+X)=\sum_{d=0}^n p_{d}(X)t^{d}$$for suitable polynomials $p_d$. Then by Taylor series we have $p_d(X)=\frac{1}{d!}\De_I^{d} \det(X)$. On the other hand, we can express $p_d(X)$ as the elementary symmetric polynomial of degree $n-d$ in the zeros of $\det(tI-X)$. Therefore, we have $p_d(X)=\sigma_{n-d,n}(\lambda(X))$ and the hyperbolicity cone of $p_d$ is the set of all symmetric matrices $A$ such that $\lambda(A)$ is in the hyperbolicity cone of $\sigma_{n-d,n}$. Thus the claim follows from \Cref{cor:elemspec}.
\end{proof}

\begin{cor}
 Let $h=\det A(x)\in\R[x_1,\ldots,x_n]$ where $$A(x):=x_1A_1+\ldots+x_nA_n$$ for real symmetric matrices $A_i$ with the property that $A(e)$ is positive definite. Then the hyperbolicity cone of every derivative $\De_e^d h$ is spectrahedral.
\end{cor}

\begin{proof}
 After replacing $A(x)$ by $S^tA(x)S$ for a suitable invertible matrix $S$, we can assume that $A(e)=I$. Then the claim follows from \Cref{cor:charapoly} by restricting to the subspace spanned by $A_1,\ldots,A_n$.
\end{proof}

\section{Wronskian polynomials and Newton's inequalities}\label{sec:newton}
Let $h\in\R[x_1,\ldots,x_n]$ be a square-free homogeneous polynomial which is hyperbolic with respect to $e\in\R^n$. It was observed in \cite[Thm.~3.1]{multiaffine} that the hyperbolicity cone $\Co (h,e)$ can be described as a linear section of the cone of nonnegative polynomials: 
$$\Co(h,e)=\{a\in\R^n:\, \Delta_{e,a}(h)\geq0\textrm{ on }\R^n\}$$
where $\Delta_{e,a}(h)=\De_eh\cdot\De_ah-h\cdot\De_e\De_ah$ is the \emph{Wronskian polynomial}. It was further shown in \cite[Thm.~4.2]{multiaffine} that if $h=\det A(x)$ where $$A(x)=x_1A_1+\ldots+x_nA_n$$ for real symmetric matrices $A_i$ with $A(e)$ positive definite, then we even have $$\Co(h,e)=\{a\in\R^n:\, \Delta_{e,a}(h)\textrm{ is a sum of squares of polynomials}\}.$$We will show in this section that this is also true for the derivatives $\De_e^d h$. More precisely, we will show that the matrices in the spectrahedral representation of the hyperbolicity cone of $\De_e^d h$ can serve as \emph{Gram matrices} for $\Delta_{e,a}(\De_e^d h)$.

Let $X$ be the generic symmetric $n\times n$ matrix and $M_1,\ldots,M_N$, $N=\binom{n}{d}$, the symmetric $d\times d$ minors of $X$. We complete $M_1,\ldots,M_N$ to an orthonormal basis $M_1,\ldots,M_r$ of $\Min_{d,n}$ (with respect to a suitable $\Or(n)$-invariant scalar product) and let $M=(M_1,\ldots,M_r)^t$. Let $\Phi:\Sym_2(\R^n)\to\Sym_2(\Min_{d,n})$ be the $\Or(n)$-linear map that we get from \Cref{cor:elemspec}. In the following, we identify $\Phi(X)$ with its representing matrix with respect to the basis $M_1,\ldots,M_r$. Finally, we denote $P_d(X)=\sigma_{d,n}(\lambda(X))$, which is a homogeneous polynomial in the entries of $X$ satifying $P_d(X)=\frac{1}{(n-d)!}\De_I^{n-d}\det(X)$. The polynomial $P_d(X)$ can also be described as the sum of all symmetric $d\times d$ minors of $X$.

\begin{lem}\label{lem:kernelmin}
 For all $A\in\Sym_2(\R^n)$ we have that $$\Phi(A)\cdot M(A)=w\cdot P_{d+1}(A)$$where $w$ is the vector whose first $N$ entries are $1$ and all other entries are $0$.
\end{lem}

\begin{proof}
 Let $S\in\Or(n)$ such that $SAS^t$ is the diagonal matrix $\Lambda$ with diagonal $\lambda\in\R^n$. We denote by $\rho(S)$ the representing matrix of the linear action of $S$ on $\Min_{d,n}$ with respect to the orthonormal basis $M_1,\ldots,M_r$. Note that $\rho(S)$ is an orthogonal matrix. By construction we have $$\Phi(SAS^t)=\Phi(\Lambda)=\begin{pmatrix} \tilde{B}(\lambda) & 0 \\ 0 & C(\lambda)                                                                                                                                                                                                                                                                    \end{pmatrix}$$where $\tilde{B}(\lambda)$ is the matrix from \Cref{rem:explicitematrix} and $C(\lambda)$ some other real symmetric matrix. Then, further using the notation of \Cref{rem:explicitematrix}, we have $$\Phi(\Lambda)\cdot M(\Lambda)=\begin{pmatrix} \tilde{B}(\lambda) & 0 \\ 0 & C(\lambda)                                                                                                                                                                                                                                                                    \end{pmatrix}\cdot \begin{pmatrix}\tilde{m}(\lambda) \\ 0 \end{pmatrix}=w\cdot \sigma_{d+1,n}(\lambda(A)).$$Because $\Phi$ is $\Or(n)$-linear, we  obtain $$\rho(S)\cdot\Phi(A)\cdot\rho(S)^t\cdot M(\Lambda)=w\cdot P_{d+1}(A).$$Since  $M_1+\cdots+M_N$ is invariant under $\Or(n)$, we have $\rho(S)^t\cdot w=w$ which shows\[\Phi(A)\cdot M(A)=\Phi(A)\cdot\rho(S)^t\cdot M(\Lambda)=\rho(S)^t\cdot w\cdot P_{d+1}(A)=w\cdot P_{d+1}(A).\qedhere\]
\end{proof}

From this we can deduce the main result of this section.

\begin{thm}\label{thm:wron}
 For all $A\in\Co(P_{d+1},I)^\circ$ we have ${P_{d+1}(A)},{P_d(A)}>0$ and  $$\De_AP_{d+1}(X)\cdot P_d(X)-P_{d+1}(X)\cdot\De_AP_d(X)-\frac{P_{d+1}(A)}{P_d(A)}\cdot P_d(X)^2$$ is a sum of squares of polynomials in the entries of $X$.
\end{thm}

\begin{proof}
 The first claim is clear since $A$ is in the interior of the hyperbolicity cones of both $P_{d}$ and $P_{d+1}$. In order to prove the second claim, we proceed as in \cite[p.~261]{kumnaplau}. By \Cref{lem:kernelmin} we have that $$\Phi(X)\cdot M(X)=w\cdot P_{d+1}(X)$$where $w$ is the vector whose first $N$ entries are $1$ and all other entries are $0$. Taking the derivative in direction $A$ of both sides gives us $$\Phi(A)\cdot M(X)+\Phi(X)\cdot \De_AM(X)=w\cdot\De_A P_{d+1}(X).$$Multiplying from the left by $M(X)^t$ and another application of \Cref{lem:kernelmin} gives:$$M(X)^t\cdot\Phi(A)\cdot M(X)+(w\cdot P_{d+1}(X))^t\cdot \De_AM(X)=M(X)^t\cdot w\cdot \De_AP_{d+1}(X).$$Since $M(X)^t\cdot w=P_d(X)$ we obtain the identity
 $$\De_AP_{d+1}\cdot P_d-P_{d+1}\cdot\De_AP_d=M(X)^t\cdot\Phi(A)\cdot M(X).$$Finally, subtracting $\frac{P_{d+1}(A)}{P_d(A)}\cdot P_d(X)^2=\frac{P_{d+1}(A)}{P_d(A)}\cdot M(X)^t\cdot w\cdot w^t \cdot M(X)$ we get that $$M(X)^t\cdot(\Phi(A)-\frac{P_{d+1}(A)}{P_d(A)} \cdot w\cdot w^t)\cdot M(X)$$is the polynomial in question. It therefore suffices to show that the matrix $$\Phi(A)-\frac{P_{d+1}(A)}{P_d(A)} \cdot w\cdot w^t$$ is positive semidefinite. Since $w\cdot w^t$ is of rank one and since $\Phi(A)$ is positive semidefinite, the polynomial $$\det(\Phi(A)-t\cdot w\cdot w^t)\in\R[t]$$ has at exactly one zero $t_0\geq0$. Moreover, the matrix $\Phi(A)-\lambda \cdot w\cdot w^t$ is positive semidefinite for all $\lambda\leq t_0$. It thus suffices to show that $t_0=\frac{P_{d+1}(A)}{P_d(A)}$. We have \[(\Phi(A)-\frac{P_{d+1}(A)}{P_d(A)} \cdot w\cdot w^t)\cdot M(A)=w\cdot(P_{d+1}(A)-\frac{P_{d+1}(A)}{P_d(A)} \cdot P_d(A))=0.\]Thus $\frac{P_{d+1}(A)}{P_d(A)}$ is a zero of  $\det(\Phi(A)-t\cdot w\cdot w^t)$.
\end{proof}

\begin{cor}[Newton's inequalities for matrices]\label{cor:newtonmatrix}
 The polynomial $$\left(\frac{P_{d}(X)}{\binom{n}{d}}\right)^2-\left(\frac{P_{d+1}(X)}{\binom{n}{d+1}}\right)\cdot \left(\frac{P_{d-1}(X)}{\binom{n}{d-1}}\right)$$ is a sum of squares of polynomials in the entries of $X$.
\end{cor}

\begin{proof}
 This is \Cref{thm:wron} for $A=I$.
\end{proof}

\begin{cor}[Classical Newton's inequalities]\label{cor:newton}
 The polynomial $$\left(\frac{\sigma_{d,n}(x)}{\binom{n}{d}}\right)^2-\left(\frac{\sigma_{d+1,n}(x)}{\binom{n}{d+1}}\right)\cdot \left(\frac{\sigma_{d-1,n}(x)}{\binom{n}{d-1}}\right)$$ is a sum of squares of polynomials in $x_1,\ldots,x_n$.
\end{cor}

\begin{proof}
 This is restricting \Cref{cor:newtonmatrix} to diagonal matrices.
\end{proof}

\begin{rem}
 Since $\binom{n}{d}^2\geq\binom{n}{d+1}\binom{n}{d-1}$ our \Cref{cor:newton} also implies that$${\sigma_{d,n}(x)}^2-{\sigma_{d+1,n}(x)}\cdot {\sigma_{d-1,n}(x)}$$is a sum of squares which was previously shown in \cite[Prop.~6]{wronsos}.
\end{rem}

\begin{cor}\label{cor:sos}
 Let $h=\De_e^k\det A(x)\in\R[x_1,\ldots,x_n]$ where $$A(x):=x_1A_1+\ldots+x_nA_n$$ for real symmetric matrices $A_i$ with the property that $A(e)$ is positive definite. Then the polynomial $$\Delta_{e,a}(h)-\frac{h(a)}{\De_eh(a)}\cdot (\De_eh)^2$$ is a sum of squares for all $a\in\Co(h,e)$ with $\De_eh(a)\neq0$. In particular, the Wronskian $\Delta_{e,a}(h)$ is a sum of squares for all $a\in\Co(h,e)$.
\end{cor}

\begin{proof}
 After replacing $A(x)$ by $S^tA(x)S$ for a suitable invertible matrix $S$, we can assume that $A(e)=I$. Then the claim follows from \Cref{thm:wron} by restricting to the subspace spanned by $A_1,\ldots,A_n$ and the fact that the cone of sums of squares is closed.
\end{proof}

\begin{rem}\label{cor:ineq1}
 For arbitrary hyperbolic polynomials the polynomial in \Cref{cor:sos} is still globally nonnegative. This can be regarded as a strengthening of the correlation inequality for hyperbolic polynomials. The following proof of this observation has been pointed out to us by  an anonymous referee after we gave a more complicated argument in a previous version of this article. Let $h\in\R[x_1,\ldots,x_n]_d$ be hyperbolic with respect to $e\in\R^n$ and $a\in\Co(h,e)$. Then for all $x\in\R^n$ the zeros of $f(t):=h(te+x)\in\R[t]$ are \emph{interlaced} by those of $g(t):=\De_a h(te+x)\in\R[t]$ in the sense that if $\alpha_1\leq\cdots\leq\alpha_d$ are the roots of $f$, and $\beta_1\leq\cdots\leq\beta_{d-1}$ the roots of $g$, we have $\alpha_i\le\beta_i\le\alpha_{i+1}$  for all  $i=1,\dots,d-1$, see e.g. \cite[Thm.~3.1]{multiaffine}. This implies that the \emph{B\'ezout matrix} $B(f,g)$ is positive semidefinite \cite[\S 2.2]{Krein81}. The entries $b_{ij}$ of the B\'ezout matrix $B(f,g)$ satisfy the identity $$\frac{f(s)g(t)-f(t)g(s)}{s-t}=\sum_{i,j}b_{ij}\cdot s^{i-1}t^{j-1}$$and a straightforward calculation shows that $$b_{11}=\Delta_{e,a}(h)(x),\,\, b_{1d}=b_{d1}=h(e)\cdot\De_a h(x),\,\, b_{dd}=h(e)\cdot\De_a h(e).$$ Since $B(f,g)$ is positive semidefinite, the minor $b_{11}b_{dd}-b_{1d}^2$ is nonnegative. This implies the inequality $\Delta_{e,a}(h)\geq \frac{h(e)}{\De_a h(e)}\cdot (\De_a h)^2$.
\end{rem}

\bibliographystyle{alpha}
\bibliography{biblio}
 \end{document}